\documentclass{amsproc}
\usepackage{epsfig,latexsym,amsfonts,amssymb,amsmath,amscd}

\newtheorem{theorem}{Theorem}[section]

\newtheorem{lem}[theorem]{Lemma}

\theoremstyle{definition}
\newtheorem{definition}[theorem]{Definition}
\newtheorem{example}[theorem]{Example}

\newtheorem{prop}[theorem]{Proposition}
\newtheorem{proposition}[theorem]{Proposition}

\theoremstyle{remark}
\newtheorem{remark}[theorem]{Remark}
\newtheorem{rem}[theorem]{Remark}
\newtheorem{ques}[theorem]{Question}
\numberwithin{equation}{section}
\newtheorem{cor}[theorem]{Corollary}

\newtheorem{MNote}[theorem]{MAJOR NOTE}

\def\a{\alpha}
\def\id{\operatorname{id}}
\def\op{\operatorname{op}}

\def\adj{\operatorname{adj}}

\def\mcA{\mathcal A}
\def\mcF{\mathcal F}
\def\val{\operatorname{val}}

\def\mcB{\mathcal B}
\def\mcP{\mathcal P}
\def\mcH{\mathcal H}

\def\mcS{\mathcal S}

\def\mcI{\mathcal I}
\def\mcAz{\mathcal A_0}
\def\mcM{\mathcal M}
\def\mcN{\mathcal N}

\def\mcS{\mathcal S}
\def\mcV{\mathcal V}
\def\C{\mathbb C}
\def\R{\mathbb R}
\def\nsets{\mathcal P^*}
\def\Cong{{\operatorname{\Phi}}}
\def\congg{{\operatorname{cong}}}

\newcommand{\etype}[1]{\renewcommand{\labelenumi}{(#1{enumi})}}
\def\eroman{\etype{\roman}}

\newcommand{\Net}{\mathbb N}
\newcommand{\Q}{\mathbb Q}
\newcommand{\Z}{\mathbb Z}

\def\supp{\mathrm{supp}}

\newcommand{\Det}[1]{ \left|{#1}\right|}

\newcommand{\trop}[1]{\mathcal{#1}}

\newcommand{\tG}{\trop{G}}

\newcommand{\one}{\mathbf 1}
\newcommand{\zero}{\mathbf 0}
\newcommand{\Hzero}{\mathbf 0_\mathbf H}

\newcommand{\tT}{\trop{T}}
\def\tTz{\trop{T}_0}

\def\bfa{{\mathbf a}}

\def\bfb{{\mathbf b}}

\def\tr{\mathrm{tr}}

\def\bfv{{\mathbf v}}

\def\ctw{\cdot_{\operatorname{tw}}}
\def\Rmax{\mathbb R}
\def\Smax{\mathbb S}

\def\ZZ{\mathbb Z}

\def\cocoa{{\hbox{\rm C\kern-.13em o\kern-.07em C\kern-.13em o\kern-.15em A}}}

\newcommand{\A}[1]{{\rm A#1}}

\def\la{\lambda}

\def\w2M{\bigwedge^2M}

\def\w{\wedge }

\def\SL{\operatorname{SL}}
\def\diag{\operatorname{diag}}

\def\Z{\ZZ}
\protect
\protect
\protect \protect

\protect
\protect

\def\be{\begin{equation}}
\def\ee{\end{equation}}
\def\bclm{\begin{claim}}
\def\eclm{\end{claim}}
\def\beqn{\begin{eqnarray}}
\def\eeqn{\end{eqnarray}}
\def\beqn*{\begin{eqnarray*}}
\def\eeqn*{\end{eqnarray*}}

\numberwithin{equation}{section}
\begin{document}

\title{Semirings}

\author{Louis Halle Rowen}

\address{Department of Mathematics, Bar-Ilan University, Ramat-Gan 52900,
Israel}
\curraddr{Department of Mathematics, Bar-Ilan University, Ramat-Gan 52900,
Israel}
\email{rowen@math.biu.ac.il}


\subjclass[2020]{Primary 16Y20,  16Y60;
Secondary 12K10, 14T10,   06A12,  20M25.}
\date{March 17, 2025}

\thanks{Thanks to Uzi Vishne for an in-depth reading of the paper and suggestions for improved exposition, and also Tomer Bauer and Be'eri Greenfeld for helpful comments}

\keywords{dependence, function semiring, hyperfield, hyperpair,  Krasner, linear algebra, magma, matrix, metatangible, pair, module, submodule, prime spectrum,  Property N, polynomial, quotient hyperfield, residue semiring, semiring, supertropical, root, tensor, vector space.}\begin{abstract}   We survey  theory developed over the past 10 years of semirings which need not be additively cancellative. The main features are a specified ``null ideal'' $\mcA_0$ of a semiring $\mcA,$   taking the place of a zero element, and a ``surpassing relation,'' taking the place of equality, which permit generalizations of the classical algebraic theory to polynomials and their roots, algebraic geometry, matrices, linear algebra,  varieties, categories, and module theory. The  ``pair'' $(\mcA,\mcA_0)$ is studied  along the lines of universal algebra.
\end{abstract}
\maketitle

\tableofcontents




\section{Overview}
 A \textbf{semigroup} is a set with an associative operation, often denoted as concatenation.   Occasionally, we use
\textbf{partial semigroups} which mean that the operation need not be defined on every pair of elements, but $(a_1 a_2)a_3 = a_1 (a_2a_3)$ whenever either side is defined. 

 An \textbf{additive semigroup} is a commutative semigroup, with the operation denoted by ``$+$," endowed with a neutral element~$\zero.$ 
 The \textbf{trivial} additive semigroup is $\{\zero\}$, and the two additive semigroups of order 2 are $\Z _2$ and the \textbf{Boolean semigroup} $(\mathbb B,+)$, where $1+1 = 1$.
An additive semigroup $\mcA$ is  \textbf{idempotent} (also called a \textbf{band})  if $a+a =a$  for all $a\in \mcA.$ Thus, $(\mathbb B,+)$ is idempotent.

  By \textbf{monoid} we mean a semigroup under multiplication, containing $\one.$ Likewise for \textbf{partial monoid}.
There is a considerable literature on semigroups,  inspired in part by the example of the multiplicative monoid of a ring. If we want to permit two operations we are led to the following definition.

A \textbf{semiring}  $(\mcA,+,\cdot,\zero,\one)$ is a set $\mcA $ with multiplication~$\cdot$ distributing over addition $+$, and having distinguished elements $\zero \ne \one$ such that $(\mcA, +,\zero)$ is an additive semigroup and $(\mcA, \cdot,\one)$ is a monoid, with $\zero$ multiplicatively absorbing, in the sense that $\zero \cdot a = a \cdot\ \zero = \zero$ for all $a\in \mcA.$\footnote{If need be, one easily can  adjoin $\zero$ formally. There also is an easy construction for adjoining~$\one$, given in  \cite[p.~3]{golan92}.}

 Semiring theory starts with the natural numbers $\Net$, and  semirings were introduced formally by Vandiver~\cite{Va}.
Golan's book \cite{golan92} is still an excellent resource, and \cite[Chapter 1]{golan92} has a wealth of examples. Semirings are a useful structure in computer programming, cf.~\cite{LMRS} since    in programming one moves forward but cannot assume to be able to move backwards.

There is a thriving theory of semigroups and monoids, but much less so with semirings.
A  satisfactory categorical theory for monoids is provided in \cite{CHWW}, but it seems that the presence of the extra operation of addition without negation often complicates the situation and hampers basic categorical results.

Semirings can be studied as $\Omega$-algebras in universal algebra \cite[Chapter 2]{Jac}. In particular, one can define the direct product of semirings $\prod_{i\in I}\mcA_i.$
One calls two structures {\it isomorphic} (written $\cong$) when there is a 1:1 correspondence which respects their operations and distinguished elements (in the case of semirings, $ +,\cdot, \zero, \one$).

However, when one wants to dig deeper, lack of negation makes it difficult to utilize roots of polynomials and determinants of matrices, two of the most important tools of algebras.

Another important area of ring theory is through the study of modules. Although one can define projective and injective modules over semirings, categorically without difficulty, it is much harder to obtain basic  results in the homology theory of  modules over semirings.

Since the theory of semirings   has  lagged far behind ring theory, researchers have compensated by considering
 \textbf{semifields} \cite{VeC1,VeC2,HuW}, which are semirings $\mcA$ for which $(\mcA\setminus \{\zero\},\cdot)$ is a group; then  group theory can be  applied to the multiplicative structure instead of the additive structure. However this is rather restrictive when one aims for a general structure theory of semirings.

This article describes an attempt to find a general  algebraic framework developed  for  semirings and related structures, such as hyperfields, often used in conjunction with tropical algebra/geometry.  The initial idea can be found in  \cite{Dr,Gau}, which was formulated as supertropical algebra in \cite{zur05TropicalAlgebra,IR}  and was implemented
for blueprints in \cite{Lor1,Lor2}, and put in a more general context in 2016 in~\cite{Row22}, in terms of a ``negation map.''  More recently this approach was  broadened further in~\cite{JMR3,AGR2} to   a ``pair'' satisfying ``Property~N,'' which still is   sufficient often to provide a robust structure theory, and we shall study such pairs in this paper. Examples treated in \S\ref{ex11}   include supertropical pairs, paired tropical extensions, hyperpairs, function pairs, polynomial pairs, and matrix pairs. Categorically we shall see that actually there  are three different types of morphisms, each giving rise to its own structure theory.

 While studying semirings, we have encountered a fascinating connection between multiplication and addition, where the semiring has an underlying multiplicative monoid on which the additive structure acts, and this will drive forward most of the theory presented here.

Some of the paper is expository, and it will be noted where considerable work remains to be done. Other parts are largely new.

 \subsection{Additively cancellative semirings}$ $

The semiring $\Net$ is one of the most fundamental mathematical entities. Of course $\Net$ can be embedded into the ring $\Z$ via Grothendieck's celebrated construction. Nevertheless, many famous assertions in the literature are formulated for $\Net$ rather than $\Z,$ for example  Lagrange's theorem that every natural number is the sum of four squares, and Waring's problem in which    $7$ is written as a sum of seven cubes $(1^3 +\dots + 1^3)$ rather than two cubes $(2^3 +(-1)^3).$

The semiring  $\Net$ can be generalized in several directions, as the set of:
\begin{enumerate}
    \item  positive elements of an ordered ring;
    \item  sums of squares in a commutative ring;
     \item  sums of $m$-powers in a commutative ring, for any $m;$
       \item   $p$-powers in a commutative ring of characteristic $p$, for any prime  $p.$
\end{enumerate}

\begin{MNote}
    In each of these examples, the semiring is \textbf{additively cancellative} ($a+b_1 = a+b_2$ implies $b_1 = b_2$);  this property is inherited from the overlying ring.
One thereby is led to try to study an arbitrary additively cancellative semiring  by embedding it  into a ring,  first embedding the additive semigroup into a group and  then naturally extending multiplication when possible. In this paper,   our emphasis will be on
non-additively cancellative semirings.
\end{MNote}

For technical reasons to become clear later, we weaken the notion of semiring:
\begin{definition}\label{nds}
   An  \textbf{nd-semiring} is an additive semigroup which also is a monoid under  multiplication, but multiplication need not   distribute over addition.
\end{definition}
\subsection{Idempotent semirings}\label{na}$ $

\begin{example}\label{gro} Vandiver's example \cite[p.~516]{Va}, denoted $\mathbb V_n$, in which  $\Net$ is truncated at some number $n,$ identifying $m$ with $n$ for all $m\ge n.$ Additive cancellation clearly fails, since $n+1 = n+0 = n.$

    A more sophisticated example is the set of finitely generated modules over a ring, where the sum is the direct sum, and the product is the tensor product.
\end{example}

But we are mainly interested in the other extreme. We say that a semiring $\mcA$ is  \textbf{idempotent} if the semigroup $(\mcA,+)$ is idempotent.

For example, the power set $(\mathbf{P}(S),\cup,\cap, \emptyset, S)$  of a set $S$, is an idempotent semiring. (Note that one could reverse $\cup$  and $\cap$ and still have an idempotent semiring!)
Similarly, the set of ideals of a ring is an idempotent semiring, under the usual addition and multiplication of ideals.

\begin{rem} $ $\begin{enumerate}\eroman
    \item An  idempotent semiring cannot be additively cancellative, since $a+\zero = a = a+a.$

    \item  Every idempotent semiring is  \textbf{zero sum free} (ZSF) \cite[p.~4]{golan92}, in the sense that $a_1 + a_2 = \zero$ implies $a_1 = a_2 = \zero$, since $a_1 = a_1 + (a_1 + a_2) = a_1 + a_2 =\zero.$\footnote{ZSF semirings are called ``antirings'' in \cite{DO} and ``lacking zero sums'' in \cite{IKR}.}
\end{enumerate}
\end{rem}

\medskip  \subsubsection{The max-plus algebra}$ $

In the \textbf{max-plus  algebra} over the rational numbers,   the sum of two numbers is taken to be their  maximum, and the product is taken to be their original sum.
The ``max-plus'' algebra  is \textbf{bipotent} in the sense that $a+b = b$ whenever $a\le b$.

More generally,  any ordered abelian  monoid $(\tG,+)$  gives rise to     a bipotent semiring $\tG \cup \{\zero\}$, with multiplication taken to be  the original addition of $\tG,$ and addition taken to be the maximum ($\zero$ formally taken to be the minimal element, which absorbs in multiplication).

The most basic example is the \textbf{Boolean semifield} $\mathbb B: =\mathbb V_2 =\{0,1 \}$ (taking the trivial monoid $\tG = \{1\}$), contained canonically in every idempotent semiring.

\medskip  \subsubsection{Connection to tropical geometry}$ $

As  explained in \cite{Gat}, the max-plus algebra is the foundation of tropical geometry, obtained via a limiting process of other semirings.
Namely, for any real number~$m,$ define a semiring structure on $\R$ by $a \bigodot_m b = a+b$ and  $a \bigoplus_m b = {\log_m ( m^a+  m^ b)}.$\footnote{A slightly different approach is to take $a \bigoplus_m b $ to be $\root m \of {a^m + b^m}$.} Then $\lim_{m\to \infty} a \bigoplus_m b = \max\{a,b\}. $

The connection to tropical geometry is as follows:
Given an algebraic curve over $\C,$ one obtains an ``amoeba'' by taking the logarithm of the absolute values of the coordinates to the base $m,$ and  the limiting case for $m\to \infty$ looks like a stick figure,  called a ``tropical curve.''  Tropicalization greatly simplifies the geometry, although preserving key enumerative properties, and tropical mathematics has become a very active subject, as seen in \cite{IMiSh,MaSt}.
An excellent introduction to algebraic aspects of semirings in tropical theory is found in~\cite{AGT}.

With the advent of the accompanying $\mathbb
F_1$-algebra/geometry, cf.~\cite{CC}, researchers have been led to tackle the theory of idempotent semirings (and semifields).

\begin{MNote}\label{draw}
    This revolutionary idea has several serious drawbacks:

\begin{itemize}
  \item In order to use logarithms, one has to pass through $\R$, which is not algebraically closed.
  \item The limit procedure of having the base of logarithms go to infinity, although conceptually clear, is difficult for computations.
    \item The lack of negatives hampers  the algebraic development of the theory, for example in linear algebra.
\end{itemize}
\end{MNote}

\subsection{Valuations and the supertropical semiring}\label{val}$ $

The first two deficiencies of Note~\ref{draw} are overcome by utilizing the field $\mathbb{K}$ of
Newton-Puiseux series $f=\sum_{\alpha \in \Q} f_\alpha t^\alpha$, where  the coefficients of $f_\alpha$ are in $\C$ (or any algebraically closed field) and the exponents of $t$ in of $f_\alpha$ are all in $\frac{1}{m}\Z$ for suitable~$m$ depending on $f.$ The field
$\mathbb{K}$ is known to be algebraically closed, so in mathematical logic is elementarily equivalent to $\C.$ At the same time we have the \textbf{Puiseux series valuation} $v: \mathbb{K}\to \Q$ taking
$f$ to $\alpha$, where $\alpha$ is  the lowest exponent of a nonzero summand $ f_\alpha t^\alpha$ of $f.$ We recall that a (nonarchimedian) \textbf{valuation} from an integral domain $K$ to an ordered abelian group $(\tG,+ \, )$
 is a multiplicative monoid homomorphism  $v: K\setminus \{ 0 \}\to \tG
$, i.e., with  $v(ab) = v(a) + v(b), $
and   satisfying the   property $v(a+b) \ge \min \{v(a),v(b) \}$ for
all $a,b \in K.$
 We formally put $v(0) = \infty.$
 Information about a valuation $v :K \to \tG \cup \{ \infty \}$ can be
garnered from the \textbf{value group} $v(K)$,
  which can replace   $\tG$.

A well-known fact about   valuations: If $v(a_1), \dots, v(a_m)$ are distinct,  then $v(\sum _{i=1}^m
a_i) = \min \{ v(a_i): 1 \le i \le m \}.$ Thus, $-v$ yields   $\max$  when $v(a_1), \dots, v(a_m)$ are distinct.
For example, if $f = 3 t ^2 + 9t ^4$ and  $g = 5t ^3 +
7t ^4$ then $f+g =  3 t ^2 + 5t ^3 + 16t ^4$ and $-v(f+g) =
-2 = \max \{ -2,-3\} = \max \{ -v(f), -v(g)\}.$ But if  $f = 3 t ^2 + 9t ^4$ and  $g = -3t^2 + 7t ^3 + 7t ^4$ then $f+g =   7t
^3 + 16t ^4$ and $-v(f+g) = -3 <-2.$  In general, we do not know a priori $v(f+g)$ when $v(f)=v(g).$ This led Zur Izhakian \cite{zur05TropicalAlgebra} to the following notion:

 \medskip  \subsubsection{The supertropical semiring}\label{st}$ $

Suppose $\tG$ is an ordered abelian  monoid with absorbing minimal element $\zero_{\tG}$, and $\tT$ is an abelian monoid, together with a monoid homomorphism $\mu: \tT\to \tG$.
Take the action $\tT\times \tG\to \tG$ defined by $a\cdot g = \mu(a)g.$
Then $\mcA := \tT\cup  \tG$  becomes an abelian monoid  when we extend the given multiplications on $\tT$ and on $\tG,$ also using the above action of $\tT$ on $\tG$.

Setting $\mu(g)=g$ for all $g\in \tG,$
we  define addition on $\mcA$ by $$b_1+b_2 = \begin{cases}
 b_1 \text{ if } \mu(b_1)>\mu(b_2),\\
b_2 \text{ if } \mu(b_1)<\mu(b_2),\\
\mu(b_1) \text{ if } \mu(b_1)=\mu(b_2).
\end{cases}.$$

We call $\mcA$ the \textbf{supertropical semiring}, denoted $\mathbb T(\tT;\tG)$.  $\mcA$ is a  semiring, but is not idempotent since $\one +\one = \mu(\one) \ne \one.$ We shall deal with this phenomenon shortly.
Also, $\mcA$  not a semifield since $\tG$ is a proper ideal.  Note that $\mcA$ contains the \textbf{supertropical Boolean semiring} $\mcB = \{\zero,\one,\mu(\one)\},$ where $\one +\one = \mu(\one).$

\medskip  \subsection{Tropical extensions}$ $

Occasionally it is convenient to refine the supertropical semiring even further.
 \begin{definition}\label{AGGGexmod1}[``Tropical extensions''  \cite[Proposition
2.12]{AGG2}, a special case of ``$\mathcal L$-layered semirings'' in  {\cite[\S 7.2]{AGR2}} and \cite{IKR0}]
  We are given a semigroup  $(\mathcal L,+,\zero)$ and  an ordered set  $\tG$.
The \textbf{tropical extension} $ \mathcal L \rtimes \tG$ of $\mathcal L$
consists of the set $\mathcal L\times \tG$, made into a semigroup with the following addition:
   \begin{equation}
\label{basicex17}(\ell_1,b_1) + (\ell_2,b_2) = \begin{cases}
(\ell_1,b_1) \text{ if } b_1 > b_2,
 \\ (\ell_2,b_2) \text{ if } b_1 < b_2,  \\  (\ell_1 + \ell_2,\, b_1)
 \text{ if }   b_1 =  b_2  .
\end{cases}\end{equation}

When $\mathcal L$ is a semiring (resp.~nd-semiring, cf.~Definition~\ref{nds}) and $\tG$ is an abelian ordered semigroup, we make $ \mathcal L \rtimes \tG$ into a semiring (resp.~nd-semiring) under the componentwise operation $(\ell_1,b_1) (\ell_2,b_2) = (\ell_1\ell_2,b_1 + b_2)$.
\end{definition}

\begin{rem}\label{Vlab}  $\mathcal L$ often is taken to be the Vandiver monoid $\mathbb V_n$ of Example~\ref{gro}.
  For $n=2,$ and $\tG$   the max-plus algebra, there is an isomorphism from $ \mathcal L \rtimes \tG$ onto the supertropical semiring $\mathbb T(\tG;\tG)$.
\end{rem}


  \subsection{Graded semirings}

\begin{definition}\label{grd}
    A semiring $\mcA$ is \textbf{graded} by a  semigroup $M$ if we can write $\mcA = \prod_{x \in M} A_x$ where $A_xA_y \subseteq A_{x+y}$. Some   cases of especial note:
    \begin{itemize}
        \item (The ``supercase'') $\mcA$ is graded by the group $(\Z_2,+)$.

        \item (The ``supertropical'' case)  $\mcA$ is graded by the Boolean monoid $(\mathbb B,+)$.

         \item (The ``layered'' case)  $\mcA$ is graded by $\mathbb V_n$.
    \end{itemize}
\end{definition}

Each of these cases arises prominently, but   an algebraic theory has been developed thoroughly only in the supertropical case.

\subsection{Hyperrings and Hyperfields}$ $

Viro~\cite{Vir} (anticipated by \cite{Mar,Mit}) found another way of viewing addition for the supertropical semifield. Instead of $a+a = \mu(a)$, he interpreted  $a+a$ as the multivalued sum $ \{ b: b\le a\}.$ So let us formalize multivalued sums.
Define $\nsets(\mcH) := \nsets(\mcH) \setminus \emptyset.$
We follow \cite{krasner,Vir}, as extended by \cite{AGR1}.

\begin{definition}\label{hy7}  Let  $\mathcal H$ be a set.
\begin{enumerate}
    \item We are given a commutative multivalued addition $\boxplus : \mathcal{H}\times \mathcal{H}\to \nsets (\mathcal H),$ which is \textbf{associative }   in the sense that if we
 define
\[ a \boxplus S = S\boxplus a =\bigcup _{s \in S} \ a \boxplus s,
\]
 then $(a_1
\boxplus a_2) \boxplus a_3 = a_1 \boxplus (a_2\boxplus a_3)$ for all
$a_i$ in $\mathcal H .$


    \item      $\mcH$ is viewed as the set of singletons in $\nsets (\mathcal H ),$ identifying $a\in \mathcal{H}$ with~$\{a\}$.

     \item  $(\mathcal H  ,\boxplus,\mathcal \Hzero)$  is a \textbf{hypersemigroup}
 when  $\mathcal H$ has an  element $\Hzero$, satisfying $\Hzero \boxplus a = a \boxplus \Hzero = a$ for all $a\in \nsets (\mathcal H )$. 

   \item A \textbf{hypergroup} is a hypersemigroup for  which every element $ a \in \mathcal H  $ has a unique \textbf{hypernegative} $-a \in \mathcal H  $, in the sense that  $\mathcal \Hzero \in a \boxplus
   (-a),$ which also distributes over hyperaddition, in the sense that $-(a_1\boxplus a_2) = (-a_1)\boxplus(-a_2).$ Here $(-)S $ denotes $\{-a: a\in S\}.$
 \end{enumerate}
  \end{definition}

 \begin{definition}$ $
     \begin{enumerate}
         \item
A \textbf{hypersemiring} (resp.~\textbf{hyperring}) is a hypersemigroup (resp. hypergroup) $\mathcal H $, which also is  a multiplicative monoid, $(\mcH,\cdot)$, with $\Hzero$ absorbing and  $\cdot$ distributing over hyperaddition.

 \item
A \textbf{hyperfield} is a hyperring $\mathcal H $ for which $(\mcH \setminus \{ \Hzero\},\cdot)$ is a group.
     \end{enumerate}
 \end{definition}

\begin{lem}\label{hp22} $ $
\begin{enumerate}\eroman
    \item  Any hypersemigroup $(\mathcal{H},\boxplus ,\Hzero)$ induces
     a semigroup       $(\nsets(\mcH),\boxplus ,\{\Hzero\}),$ with addition  given by $$S_1 \boxplus   S_2 = \cup \{s_1 \boxplus  s_2: s_i\in S_i\}.$$
     \item When $\mathcal{H}$ is a hypersemiring, then  $(\nsets(\mcH),\cdot,\{\one\})$ is a monoid, where we define $S_1 \cdot S_2 = \{a_1a_2: a_i \in S_i\}$, which satisfies ``single distributivity'':

    \begin{equation}
         a \cdot \boxplus  S_i = \boxplus  (a\cdot S_i)  ; \quad (\boxplus  S_i)\cdot a =  \boxplus (S_i  \cdot a), \qquad \forall a\in \mcH, \ S_i \in \nsets(\mcH).
     \end{equation}
\end{enumerate}
 \end{lem}
 \begin{proof} Associativity and single distributivity are checked elementwise.
 \end{proof}

Warning: Even when $a_1,a_2\in a_1\boxplus a_2,$  the hyperset $ S \cdot (a_1 \boxplus a_2)$ could be properly contained in  $S\cdot a_1 \, \boxplus \, S\cdot a_2,$ as we shall see below in Example \ref{ex-krasner}(iii).

   \medskip  \subsubsection{Examples of hyperfields}$ $

Hyperfields have become an active area of investigation \cite{AGR1,AGT,BaL,BaZ,BoS,CC,Ho,HoJ,krasner,Mar,Mas,MasM,Mit,NakR,Vir}.
   We recall some major examples of hyperfields from \cite{AGR1,krasner,Vir}.
Also, a~large collection of hyperrings is given in \cite[\S 2]{MasM}, and \cite{Ho}, and    related examples are in~\cite{NakR}.

\begin{example}\label{ex-krasner}$ $
\begin{enumerate}\eroman

\item   Our motivation came from the \textbf{tropical hyperfield} \cite{Mit,Vir} which
  consists of the set $\mcH=\R \cup \{-\infty\}$, with $-\infty$
  as the zero element $\Hzero$ and $0$~as the unit element $\one$, equipped
  with  addition $a\boxplus b = \{a\}$ if $a>b$,
  $a\boxplus b = \{b\}$ if $a<b$,
  and $a\boxplus a= [-\infty,a]$. It is easy to see that there is an embedding from the supertropical semifield to the tropical hyperfield, given by $a \mapsto a$ and $\mu(a) \mapsto [-\infty,a].$

  A special case is  the \textbf{Krasner hyperfield} \cite{krasner}  $ \mathcal K := \{ 0, 1 \}$, with the usual multiplication
law, and with hyperaddition defined by  $ x \boxplus  0 = 0 \boxplus  x =
x $ for all $x ,$ and $1 \boxplus  1 =  \{ 0, 1\}$.   $\nsets(\mathcal K) = \mathcal K\cup \{\mathcal K\}$ is isomorphic to~$\mathcal{B}$, by the map $0 \mapsto \zero, \ 1\mapsto \one, \ \{0,1\}\mapsto \mu(\one).$
\item
  The \textbf{hyperfield of signs}
$ \mathcal S := \{ 1 , 0, -1\}$,
  with the usual multiplication
law, and hyperaddition defined by $1 \boxplus  1 = 1 ,$\ $-1
\boxplus  -1 = -1  ,$\ $ x \boxplus  0 = 0 \boxplus  x = x $ for all
$x ,$ and $1 \boxplus  -1 = -1 \boxplus  1 = \{ 0, 1,-1\} $.
$\mathcal{S}\cup  \{\mathcal S\}$ is a semiring isomorphic to $\widehat{\mathbb B},$ by sending $0 \mapsto \zero,$ $1 \mapsto (1,0)$,  $ -1 \mapsto (0,1)$, and $\{ 0, 1,-1\} \mapsto (1,1).$

\item  The \textbf{phase hyperfield}. Let $S^1$ denote the complex
  unit circle, and take $\mcH = S^1\cup \{ 0 \}$.
  Nonzero points $a$ and $b$ are \textbf{antipodes} if $a = -b.$
Multiplication is defined as usual (corresponding on $S^1$ to
addition of angles). We denote $(a\,b)$ for an open arc of less than 180 degrees
connecting two non-antipodal points $a,b$ of the unit circle. The hypersum is given, for $a,b\neq 0$, by
$$a \boxplus b=
\begin{cases} (a\,b) \text{ if } a \ne b \text{ and } a\neq -b;\\
  \{ -a,0,a \} \text{
if } a = -b , \\   \{  a \} \text{ if }   a = b.
\end{cases}$$
  \end{enumerate}
  The phase hyperfield fails double distributivity, since, for $a\ne -b,$  $-(a\,b)\boxplus (a\,b) = \mcH $ whereas $\{-1,0,1\} (a\,b)$ is $(a\,b)\cup \{0\}\cup -(a\,b)$.
  \end{example}

\medskip  \subsubsection{Krasner's residue construction}\label{Kr1}$ $

There also is a general construction   of Krasner \cite{krasner}, done more generally in~\cite{Row24}.

 \begin{theorem}\label{theoremE}{\cite[Theorem D]{AGR2}}$ $
 \begin{enumerate}\eroman
     \item Suppose $(\mathcal S,+)$  is a semigroup and $f:\mathcal S \to \bar {\mathcal S}$ is any  set-theoretic map onto a set $\bar {\mathcal S}$. Define   hyperaddition $\boxplus: \overline {\mathcal S} \to \mathcal{P}(\overline {\mathcal S})$
  by $\overline {a} \boxplus \overline{a'} = \{ \overline{a+a'}: f(a) = \overline{a}, f(a') = \overline{a'}\},$ for $a,a'\in {\mathcal S}.$ Then $\overline {\mathcal S}$ is a hypersemigroup (i.e., is associative). If  ${\mathcal S}$ is a  group then $\overline {\mathcal S}$ is a hypergroup, where $-\overline{a}=  \overline{-a}.$
  \item Suppose a monoid $G$ acts on a semiring $(R,+)$, such that the orbits are multiplicative, i.e., $(g_1 a)( g_2 a' )= g_1g_2(aa').$ Write $  \overline { a} = Ga,$ for $a\in R.$ Then $\overline {R}:= R/G$ is a hypersemiring under the hyperaddition of (i), with
   $\overline a \overline{a '}:= \overline{aa'}.$
 \end{enumerate}
\end{theorem}

 \begin{proof} (i) As in  \cite{krasner}, given in \cite[Proposition~2.13]{JMR3}. Namely $$(\overline {a} \boxplus \overline{a'})\boxplus \overline{a''} = \{ \overline{a+a'+a''}: f(a) = \overline{a}, f(a') = \overline{a'}, f(a'')=\overline{a''}\} =\overline {a} \boxplus (\overline{a'} \boxplus \overline{a''}).$$

     (ii) First note by induction that $   \boxplus b_i  G = \{ \sum a_i G : \quad a_i   \in b_iG \}.$ Then  $$ aG(\boxplus_{i=1}^{m} b_i  G )= \left\{ aG  \sum a_i'G:\ a_i' \in b_i G\right\}    =\left\{ (\sum a  a_i')G: \,a_i' \in b_i G\right\} =  \boxplus_{i=1}^{m} aG b_i  G, $$
  implying $\overline {R}$ is a hypersemiring.
   \end{proof}

 \begin{example} The major example of (ii) is when  $G$ is a multiplicative subgroup of $R$. If $R$ is a field, then $R/G$ is a hyperfield,  known in the literature as the \textbf{quotient hyperfield}.
 \end{example}

 \medskip  \subsubsection{Description of major hyperfields as quotient hyperfields}$ $

We want to identify the  different examples of Example~\ref{ex-krasner} as quotient hyperfields. Towards this end, an \textbf{isomorphism} $f:\mcH \to \mcH'$ of hyperfields is defined as a   group isomorphism satisfying $f(a_1 \boxplus a_2) =f(a_1)\boxplus f(a_2)$ for all $a_i\in \mcH.$

\begin{example}\label{hypers} Here are isomorphisms from Examples~\ref{ex-krasner} to quotient hyperfields.
\begin{enumerate}\eroman

  \item  The {tropical hyperfield}
  is isomorphic to the quotient hyperfield $F/G$,
  where $F$ denotes any field with a surjective
  valuation $v: F \to \R\cup\{+\infty\}$, and
  $G:=\{f\in F: v(f) =0\}$.

  The {Krasner hyperfield}  is isomorphic to the quotient hyperfield $F/F^\times$, for any field $F$.

\item  The {hyperfield of signs}  is isomorphic to   the quotient hyperfield $K/K_{>0}$ for each linearly
ordered field $(K,\leq )$,
where $K_{>0}$ is the group of positive elements of $K$.

\item  The  {phase hyperfield}  is isomorphic to the quotient hyperfield
$\C/\R_{>0}$.

\end{enumerate}
\end{example}

\begin{remark}\label{hp2} Suppose that  the hypernegative $-\one$ exists in  the quotient hypermonoid $\mcH =\mcM/G$. Define $e = \one \boxplus (-\one) \in \nsets(\mcH)$ (so $\Hzero\in e$).

If $G = \{ \pm 1\}$, then $\Hzero\in a \boxplus a$ for all $a\in \mcH.$
\end{remark}

The failure of distributivity  in  some quotient hyperfields   enhances the following result.

\begin{lem}\cite[Lemma~3.5]{Row24}
  If $-1\in G$, then $ee = e\boxplus e$  in  $\mcM/G$.\end{lem}

It was thought until quite recently that non-quotient hyperfields are rare, but  large classes of  nonquotient  hyperfields have been found in \cite{Mas,MasM,Ho,HoJ}.

 \subsection{Semiring constructions}

Many standard constructions in algebra are applicable to semirings.

\subsubsection{Doubling}$ $

One way of circumventing the lack of negatives is by \textbf{doubling} a semiring~$\mcA$, inspired by Grothendieck's familiar construction of $\Z$ from $\Net$:

\begin{definition}\label{doub}
    Define the \textbf{doubled semiring} $\widehat{\mcA} = \mcA\times \mcA.$
  \textbf{Twist multiplication} on $\widehat{\mathcal A}$
 is defined as follows:

 \begin{equation}\label{twis} (b_1,b_2)\ctw (b'_1,b'_2) =
 (b_1b'_1 + b_2 b'_2, b_1 b'_2 + b_2 b'_1),\qquad (b_1,b_2), (b'_1,b'_2) \in \widehat{\mathcal A}.  \end{equation}
\end{definition}

\begin{rem}\label{doubb} The doubled semiring $\widehat{\mcA} $ is $(\Z_2,+)$-graded, the $0$ component being  $\mcA \times \{\zero\},$ and the $1$  component being  $\{\zero\}\times \mcA .$ Also
   $\widehat{\mcA} $ has an equivalence relation given by $(b_1,b_2)\simeq (b'_1,b'_2)$ if $b_1 b'_2 = b_2 b'_1 $, but we do not apply it, since it would yield a ring. Instead Gaubert~\cite{Gau},   who writes $\widehat{\Rmax}$ as~${\Smax},$ for ``symmetrized,'' viewed the first component $\mcA$ of  $\widehat{\mcA}$, denoted $\mcA^+,$ as the {\it positive part}, the second component, denoted $\mcA^-,$ as the {\it negative part}, and the diagonal
    $\{(a,a): a\in A\}$ as the {\it null part}. $\mcA^+$ is isomorphic to $\mcA$ via $(a,0)\mapsto a$ and when ${\mcA}$ is idempotent, so is $\widehat{\mcA}$.
\end{rem}

\begin{rem}\label{absa}
   The \textbf{absolute value} $||(b_1,b_2)|| := b_1+b_2$ is a homomorphism $\widehat{\mcA} \to \mcA$, which is a retract of the injection of $\mcA$ to the positive part of $\widehat{\mcA}.$
 \end{rem}

\subsubsection{Monoid semirings}

\begin{definition}
    The \textbf{monoid semiring} $\mcA[\mcM]$ of a monoid $\mcM$ over a semiring~$\mcA$ is the semigroup $\mcA^{(\mcM)}$,  under the addition of Example~\ref{dp}(i),  which  becomes a semiring   with \textbf{convolution product} $fg= ((\sum _{uv=s} f_u g_v)s),$ where $f = (f_u)$ and $g = (g_v)$.
\end{definition}

\begin{example}\label{poly} The  \textbf{polynomial semiring}. We write $\mcA[\Lambda]$ for the semiring $\mcA[\mcM]$, when $\mcM$ is the  word monoid in $t$ commuting letters $\la_1,\dots, \la_t$. The  \textbf{monomials} of $\mcA[\Lambda]$ are the polynomials with singleton support. Thus each polynomial is a sum of monomials.
The case $\Lambda = \{\la\}$ is of special interest,  denoted $\mcA[\la]$.
  The semiring $\Net[\lambda]$ plays a key role in the work of Lorscheid \cite{Lor1,Lor2} in $\mathbb{F}_1$-geometry.
\end{example}

\subsubsection{Function semigroups and semirings, cf.~\cite[Chapter 2]{golan92}}$  $

 \begin{definition} Suppose $X$ is a set with a distinguished element $\zero.$
     Write $X^S$ for the functions from a set $S$ to $X$. For $f\in X^S$ define the \textbf{support} $\supp(f) = \{ s \in S: f(s)\ne \zero\},$  and $X^{(S)}$ for the functions with finite support.
 \end{definition}

 \begin{example}\label{dp} Notation as above, \begin{enumerate}\eroman
     \item For an additive semigroup  $\mcA$,   $\mcA^S$ is an additive semigroup under elementwise addition of functions,  whose zero element is the function with empty support, i.e., sending   all   elements to~$\zero$. The elements of $\mcA^S$ can be viewed as $S$-tuples, and the sub-semigroup $\mcA^{(S)}$
  is the \textbf{semigroup direct sum}.

      (These are instances of the familiar direct product and direct sum of copies of $\mcA$ in universal algebra.)

   \item    If $\mcA$ is a semiring and $S$ is considered an index set then  $\mcA^S$ is a semiring with componentwise addition as in (i) and componentwise multiplication, where the unit element is the function sending  all   elements of $S$ to $\one.$
 \end{enumerate}
 \end{example}

 However, we are interested in a different semiring structure.

 \begin{lem}\label{fpr} $ $
   \begin{enumerate}\eroman
  \item When $\mcA$ is a semiring and $S$ is an additive semigroup, $\mcA^{(S)}$ becomes a semiring  under the addition of Example~\ref{dp}(i), but with \textbf{convolution product} $fg: s \mapsto \sum _{u+v=s} f(u)g(v).$  The unit element is the function sending $\zero \mapsto \one$ and all other elements to $\zero$.

  \item Generalizing (i), when $S$ is a partial monoid having the property that for any $s\in S$ there are only finitely many pairs $(u,v)\in S\times S : u+v =s,$ such as for $S$ well-ordered, then we can still define the convolution product.
   \end{enumerate}
\end{lem}

\medskip  \subsubsection{Function  polynomial semirings and related constructions}
 \begin{example} \label{polyc}
 There is a host of instances of the function semiring construction over a   semiring $\mcA$,  under the convolution product.
 \begin{enumerate}\eroman
     \item We take a set of indeterminates  $\Lambda= \{\lambda_i : i\in I\}$. The \textbf{function polynomial semiring }, is $\mcA^{({\Net^{(I)}})}$where $f(\Lambda)$ denotes the function   $(m_i) \mapsto (f(m_i)),$ for~$(m_i)\in \Net^{(I)}.$
 \item  The \textbf{function power series semiring } is $\mcA^{\Net}$.

 \item  The \textbf{function Laurent series semiring} is the set of functions $f\in \mcA^\Z$ for which $\supp(f) $ is bounded from below.

  \item The field $\mathbb{K}$ of
Newton-Puiseux series from \S\ref{val} is another instance.
\end{enumerate}
 \end{example}

\begin{MNote}
    There is a canonical epimorphism $\psi: \mcA[\mcM] \to \mcA^{(\mcM)}$   sending $f = (f_s)$ to the function  $s \mapsto f_s$ in $\mcA.$ But $\psi$ need not be 1:1. For example, if $\mcA$ is the max-plus algebra then the polynomials $\la^3 +\la^2 +\la$ and $\la^3   +\la$ define the same function! Actually, this situation also arises classically for polynomials over a field~$F$ of $n$~elements; the polynomials $\la^n$ and $\la$ define the same function.
\end{MNote}
\medskip  \subsubsection{Matrix semirings}$ $

The \textbf{matrix semiring} $M_n(\mcA)$ can be defined as usual, with a \textbf{set  of matrix units} $S=\{e_{i,j}:\, 1\le i,j\le n\}$ satisfying $e_{i,j}e_{k,\ell} = \delta_{j,k}e_{i,\ell}$ where $\delta$ denotes the Kronecker delta, and $\sum_{i=1}^n e_{i,i} = \one$. But matrix units are not closed under multiplication, so we need  modify Lemma~\ref{fpr}(iii) in order to view matrices as a function semiring.
Namely, switching from addition on $S$ to multiplication to make the notation more familiar, suppose multiplication is defined only on a subset of $S$.

 $\mcA^{S}$ becomes a semiring  under the {convolution product} $fg: s \mapsto \sum _{uv=s} f(u)g(v).$   Take $\mcA^{S}$ where $S = \{ (i,j): 1\le i,j\le n\}$, under multiplication $(i,j)(k,\ell)=   (i,\ell)$ when $j=k$, and is undefined otherwise;  we get $M_n(\mcA)$ by     imposing the relation $\sum e_{i,i} = \one$.


\subsection{Modules and semirings over a  monoid}$ $

Just as one studies rings as algebras over an underlying commutative ring (often a field), we might be inclined to study
semirings over an underlying commutative semiring, but it turns out that the key is its multiplicative structure, which is a monoid.

\begin{definition}\label{um}
 $(\tT, \one)$ is a    monoid with a unit element $\one$.\footnote{To handle  matrix units, we could start with a partial monoid $\tT$, see next footnote.}
\begin{enumerate}
  \item
A $\tT$-\textbf{module}  
is an additive semigroup
$(\mathcal A,+,\zero_\mcA)$ together with a  (left)   $\tT$-action $\tT\times \mathcal A \to \mathcal A$ (denoted  as
concatenation), which is
\begin{enumerate}
\item \textbf{associative},  in the sense that
$a_1(a_2b) = (a_1a_2)b $ for all $ a_i \in \tT,$ $b  \in \mathcal A$,
whenever $a_1a_2$ is defined, and $a_1(a_2b) $  is undefined whenever  $a_1a_2$ is undefined.

 \item zero absorbing, i.e. $a \zero _\mcA = \zero_\mcA, $ $ a \in \tT$.

  \item \textbf{distributive},  in the sense that
$$a(b_1+b_2) = ab_1 +ab_2,\quad \text{for all}\; a \in \tT,\; b_i \in \mathcal A.$$

\item  $\one b= b$ for all $b\in \mcA.$
 \end{enumerate}

  \item   We call $\tT$ the \textbf{underlying monoid}  of  \textbf{tangible} elements. Our convention is to write $a$ for an element of $\tT,$ and $b$ for an element of $\mcA.$  We shall assume henceforth that $\mcA $ is a torsion-free $\tT$-{module}, in the sense that
if $a b _1= ab_2$, for  $a\in \tT$ and $b_1,b_2\in \mcA$, then  $b_1=b_2.$

\item  Define  formally the monoid $\tTz = \tT \cup \{\zero\}$ with absorbing element $\zero$ with $\zero \cdot \mcAz= \{\zero_\mcA\}$.\footnote{For $\tT$ a partial monoid, define $a_1a_2= \zero$ in $\tTz$ whenever $a_1a_2$ is undefined in $\tT.$}
\item  A $\tT$-module $\mcA $ is \textbf{free} if there is a \textbf{base} $\mcB$ such that every element of $\mcA$ can be written uniquely in the form $\sum a_b b : \  b\in \mcB,\ a_b \in \tTz$.



 \end{enumerate}
\end{definition}

\begin{example}\label{frees}$ $
\begin{enumerate}\eroman
  \item Any semigroup trivially is a $\{1\}$-module.
  \item Suppose $\mcA$ is a $\tT$-module.
  \begin{itemize}
       \item For any monoid $\mcM$, $\mcA[\mcM]$ is a $\tT$-module under the   action $a(f_u) =(af_u).$

       \item For any set $S$, $\mcA^S$ is a $\tT$-module under the pointwise action $(af)(s) = af(s).$

       \item For any $\tT$-module $\mcA,$ monoid $\mcM$, $\mcA[\mcM]$ is a $\tT$-module under the   action $a(f_u) =(af_u).$
  \end{itemize}

    \item  The \textbf{free $\tT$-module} is $\Net^{(\tT)},$ viewing $\Net$ as an additive semigroup.

        \item  The underlying monoid of the doubled $\tT$-module is $\tT \times \{\zero\} \, \cup \, \{\zero\}\times \tT.  $
\end{enumerate}
\end{example}

\begin{MNote} Many important properties are not preserved under the polynomial and matrix constructions. This is evident in hyperrings, for example.
       If $\mathcal{H}$ is a hyperring then
$\mathcal{H}[\lambda]$ is not a  hyperring in any natural way, since how would we define $a_1 \lambda \boxplus a_2 \lambda^2$ and $( \one \boxplus   \lambda)^2$?
Likewise, matrix constructions do not work well with hyperrings.

The difficulty is that we are not working directly with the hyperring, but rather its power set, which may fail to be distributive (cf.~\S\ref{nddis} below). Thus representation theory over hyperfields is problematic.
\end{MNote}

 For ease of exposition, we assume $\tTz \subseteq \mcA.$

\begin{definition}$ $
    A \textbf{$\tTz$-semiring}
 is a semiring which is also a
 $\tTz$-module, such that the natural map  $\iota :\tTz\to \mcA$ given  by $\iota(a) = a \one_\mcA$ is an injection  such that $\iota(\one) = \one_\mcA$.
 (Thus,
 the $\tT$-module action and multiplication by $\iota(\tT)$ within $\mcA$ are the same.)  

  We also assume throughout this exposition that $ab = b \iota(a )$  and $b_1(a b_2) =a(b_1 b_2) = (a b_1)b_2$ for all $a\in \tT,$ $b_i\in \mcA,$ i.e., the copy of $\tT$ in $\mcA$ is central.
\end{definition}

\subsection{Homomorphisms}$ $

\begin{definition}\label{hom1} Suppose $\mcA_i$ are  $\tT_i$-modules, for $i = 1,2.$ \begin{enumerate}\eroman
   \item   Our \textbf{maps} $f$ from   $\mathcal A_1$ to   $\mathcal A_2  $  always will
 satisfy
\begin{enumerate}
    \item $f (\tT_1) \subseteq \tT_2,$

\item $f(\zero_{\mcA_1}) = (\zero_{\mcA_2}),$  and $f(\one_{\mcA_1}) = \one_{\mcA_2},$

\item   $f(ab) = f(a) f(b)$ for all~$a\in \tT_1, \ b \in \mcA_1.$

\end{enumerate}

 \item     A  \textbf{$\tT_{1,2}$-homomorphism} $f:\mcA_1 \to
    \mcA_2$ is a  map satisfying $$f(b_1+b_2) =  f(b_1)+f(b_2), \quad\forall b_i\in \mcA_1.$$

\end{enumerate}
\end{definition}

\begin{lem}\label{cb} Suppose $f:\mcA_1 \to
    \mcA_2$   is a  $\tT_{1,2}$-homomorphism.
     \begin{enumerate}\eroman
  \item $\mcA_2$ is an $f(\tT_1)$-module, so we can replace $\tT_2$ by $f(\tT_1),$ and assume that  the restriction $f_{\tT_1}$ (defined
  $f_{\tT_1}(a) = f(\iota(a))$ for $a\in \tT_1$) is onto.

      \item Assuming $f_{\tT_1}$ is onto, view $\mcA_2$ as a $\tT_1$-module, by defining $a  b_2 = f(a )b_2$ for $a \in \tT_1,$ $b_2\in \mcA_2.$

     \item Let  $\id:\mcA_2 \to \mcA_2 $ denote the identity, from $\mcA_2$ as $\tT_1$-module to $\mcA_2$ as $\tT_2$-module.
     Then
      $f$ is a composition of    $\tT_{1,1}$-homomorphisms $f:\mcA_1 \to
    \mcA_2$  as $\tT_1$-modules, satisfying $f(ab) = af(b)$ for $a\in \tT_1,$  followed by $\id.$
     \end{enumerate}
\end{lem}
\begin{proof}
    (i)  $f_\tT(aa') = f(\iota(aa')) =f(\iota(a) f(\iota(a') = f_\tT(a)f_\tT(a').$

    (ii), (iii) A standard change of the underlying monoid. Note by definition that $a f(b)=f(a)f(b) = f(ab).$
\end{proof}

\begin{rem}
    Since Lemma~\ref{cb} is canonical, we  assume from now on, for all   $\tT_{1,2}$-homomorphisms $f:\mcA_1 \to
    \mcA_2$, that $\tT_1 = \tT_2 = \tT,$ and    $f_{\tT}$ always is the identity map. Thus a  \textbf{$\tT$-homomorphism} satisfies $f(a  b)= a f(b)$ for $a\in \tT,$ $ b\in \mcA_1$.
\end{rem}


\begin{definition}
  For $\tTz$-nd-semirings $\mcA_1,\mcA_2,$ a map $f:\mcA_1\to\mcA_2$ is \textbf{multiplicative} if
$f(b_1b_2) =  f(b_1)f(b_2),$ $\forall b_i\in \mcA_1.$
   A  \textbf{$\tTz$-nd-semiring homomorphism} $f:\mcA_1 \to
    \mcA_2$ is a $\tT$-homomorphism  which also is a multiplicative map.
\end{definition}

\begin{MNote}$ $
 \begin{itemize}
     \item  As explained in \cite[Chapter 2]{Jac}, in universal algebra, homomorphisms are defined via congruences, where we recall that a congruence on $\mcA$ is an equivalence relation which viewed as a subset of $\mcA \times \mcA$ is a  subalgebra. In particular, the \textbf{trival congruence} $\diag$ is
    $\{(a,a): a\in \mcA\}$. This approach fits in well with semigroups and semirings viewed via universal algebra, but does not always yield the categories that we desire, as we shall see, and other morphisms are described in \S\ref{morp1}.

   \item   Given a congruence $\Cong$ on a semiring $\mcA,$ one defines $\mcA/\Cong$ as the semiring $\{ [b]: b\in \mcA\}$,   identifying $[b_1]=[b_2]$ when $(b_1,b_2)\in \Cong.$ (For example, $\mcA/\diag \cong \mcA.$)

     \item  If $f:\mcA_1\to\mcA_2$ is a homomorphism then for  $I\triangleleft \mcA_2,$ the inverse image $f^{-1}(I)$ is an ideal of $\mcA_1$, not a congruence. This creates a certain tension, where sometimes we choose  ideals and sometimes congruences, to be explained in Note~\ref{decomp1}.
 \end{itemize}
\end{MNote}

\section{Pairs}

One can continue to make  definitions parallel to rings,
but one encounters obstacles when attempting to factor polynomials or to describe dependence of vectors, due to the lack of negation. This leads us to the objective of this paper, to find
  some extra structure which will support a plausible structure theory. The underlying idea, originating in Gaubert's dissertation~\cite{Gau}, is to replace $\zero$, which is useless for ZSF semirings, by a designated additive subsemigroup.

The groundwork has been laid for our main notion.

 \begin{definition}\label{symsyst}
 \cite{AGR2,JMR3}. $ $ $\tT$ is the {underlying monoid}, cf. Definition~\ref{um}.
\begin{enumerate}
    \item A  $\tT$-\textbf{pair}  $(\mcA,\mcA_0)$ is  a  $\tT$-module $\mcA$, given together with a specified $\tT$-submodule  $\mcA_0$, i.e., $ab_0  \in \mcA_0$ for all $a\in \tT$ and $b_0\in \mcA_0,$ and which also satisfies the converse:

If $a b \in \mcA_0$ for $a\in \tT,$ then $b\in \mcA_0$ (which is obvious if $\tT$ is a group).

 \item From now on, excepting the Examples (iv), (v), and (vi) of \S\ref{ex11} and Definition~\ref{mapr}, the underlying monoid $\tT$ is given (until \S\ref{ext}), and   we simply call a $\tT$-pair     a ``pair.''
 We call $\mcA_0$ the ``null part'' of~$\mcA,$ which we informally view as 0.
\end{enumerate}
\end{definition}

\begin{MNote}\label{UA}$ $ \begin{enumerate}\eroman
   \item    The prototype which we call the \textbf{classical pair} and whose structure theory we  imitate, is the pair $(\mcA,\{\zero\}),$ where $\mcA$ is a faithful algebra over an integral domain $C.$
$\tT$ is the monoid $C\setminus \{0\}$ and $\tTz=C,$ and $\mcA_0 = \{0\}.$

   \item If $P$ is a prime ideal of an integral domain $\mcA,$ we have the pair $(\mcA,P)$ with
 underlying monoid  $C\setminus P$, which is a domain that provides a different perspective to the classical integral domain $\mcA/P$.

   \item  In the same spirit as (i) and (ii), we have the pair $(\mcA,\{\zero\}),$ where $\mcA$ is any semiring, and $\tT$ is any submonoid, such as $\{1\}.$

\item      The notion of ``pair'' is  general on purpose, in order to let $\mcA_0$ replace the neutral element in other diverse
algebraic situations. Following the spirit of universal algebra in~\cite{Jac}, we define an \textbf{abstract pair} $(\mcA, \mcA_0)$ to consist of an $\Omega$-algebra with an action of a  monoid   $\tT$ on $\mcA$ which preserves the universal algebra operations, and $\mcA_0$ is an $\Omega$-subalgebra closed under the $\tT$-action. For abstract pairs, we no longer are requiring that $\tT \subseteq \mcA.$
   \begin{itemize}
  \item  From this point of view, we do not need $\zero \in \mcA,$ and we could define $\mcA = \tT \cup \{\infty\}$, where $a\infty = \infty$ for all $a\in \tT.$ For example, take $\tT = \{\one\}$ and $(\mcA,\mcA_0) = ( \{ \one, \infty\},\{\infty\}). $

  \item To define \textbf{monoid pairs}, we need the conditions $a(b_1b_2)= (ab_1)b_2$ and $a\mcA_0 \subseteq \mcA_0$
for all $a\in \tT,$ $b_i\in \mcA.$

 \item Lie pairs were treated in \cite{GaR}.  \end{itemize}
 The \textbf{direct sum} $\oplus(\mcA_i,{\mcA_i}_0)$ of abstract pairs $(\mcA_i,{\mcA_i}_0)$ all having underlying monoid $\tT$ is defined as the pair $(\oplus \mcA_i,\oplus {\mcA_i}_0),$ with the diagonal action $a(b_i)= (ab_i).$

  An \textbf{isomorphism} of abstract pairs is an $\Omega$-algebra isomorphism which restricts to an $\Omega$-isomorphism of the null parts.
\end{enumerate}
\end{MNote}
Non-classical examples of pairs of main interest to us are in \S\ref{ex11} and \cite[\S 5]{AGR2}.

 \begin{definition}$ $ (We continue to  assume that $\tTz \subseteq \mcA$.)
     \begin{enumerate}
\item
A  pair $(\mcA,\mcA_0)$ is said to be
 \textbf{shallow} if $ \tT \cup \mcA_0 =\mcA  .$

\item
A pair $(\mcA,\mcA_0)$ is of the \textbf{first kind} if $\one + \one \in \mcA_0$  and of the \textbf{second kind} if $\one + \one \notin \mcA_0$ (And thus $a + a \notin \mcA_0$ for all $a\in \tT$).

     \end{enumerate}
 \end{definition}

 \begin{example}
    The classical pair $(\mcA,\{\zero\})$ is always proper, and is shallow when $\mcA = C;$ it is of the first kind when $\mcA$ has characteristic~2.
 \end{example}
\medskip

\subsubsection{Paired maps and homomorphisms}

\begin{definition}\label{congd1}$ $\eroman
Suppose $(\mcA,\mcA_0)$ and $(\mcA',\mcA'_0)$ are pairs.
\begin{enumerate}
\item
A \textbf{paired map} $\theta:(\mcA,\mcA_0)\to (\mcA',\mcA'_0)$  is a  map $\theta:\mcA \to \mcA'$ satisfying
$\theta:(\mcA_0)\subseteq \mcA_0'.$
\item
A \textbf{paired homomorphism}
is a paired  map which is a homomorphism (taken in its appropriate context, e.g., semiring pair or nd-semiring pair).

\item The  \textbf{kernel congruence}  of a paired map $\theta:(\mcA,\mcA_0)\to (\mcA',\mcA'_0)$ is $(\Cong_\theta,{\Cong_\theta}_0)$, where $\Cong_\theta := \{(b_1,b_2)\in \mcA \times \mcA: \theta(b_1) = \theta(b_2)\}$ and ${\Cong_\theta}_0 := \{(b_1,b_2)\in \Cong_\theta : \theta(b_1)\in \mcA'_0\}. $
\item  A \textbf{paired injection}  is a paired homomorphism   $\theta:(\mcA,\mcA_0)\to  (\mcA',\mcA_0')$ with $\theta^{-1}(\mcA_0')\subseteq \mcA_0$,
    together with $\theta(\tT) \subseteq \tT'$.
 \end{enumerate}
\end{definition}

\begin{lem}\label{C1} Suppose that $\Cong$  is a congruence of the pair $(\mcA,\mcA_0)$, and let $\Cong_0 = (\mcA_0 \times \mcA_0)\cap \Cong$.
    \begin{enumerate}\eroman
        \item  Then $(\mcA/\Cong,\mcA_0/{\Cong_0})$  is a pair, with the ``projection'' paired homomorphism $\pi: (\mcA,\mcA_0) \to (\mcA/\Cong,\mcA_0/{\Cong_0})$ given by $b\mapsto [b].$ 

\item If $\mcA_0 \subseteq \mcA_0' \subseteq \mcA,$ then the identity map induces a paired map $ (\mcA,\mcA_0)\to (\mcA, \mcA_0')$, which we call \textbf{an extension of the null part}.
    \end{enumerate}
\end{lem}
\begin{proof}
    (i) Standard for congruences; one just checks that $\Cong_0$ is a congruence of~$\mcA_0,$ being the restriction of $\Cong.$

    (ii) Immediate.
\end{proof}

 \begin{prop}\label{decomp}
     Any paired map  $\theta : (\mcA,\mcA_0)\to (\mcA',\mcA'_0)$ is the composite of the projection $(\mcA,\mcA_0)\to (\mcA/\Cong_\theta,\mcA_0/{\Cong_\theta}_0)$, followed by a paired injection and finally
      an extension of the null part (up to isomorphism).
 \end{prop}
 \begin{proof}
     By Lemma~\ref{C1}, there is an induced paired map $\bar \theta: (\mcA/\Cong_\theta,\mcA_0/{\Cong_\theta}_0)\to (\mcA',\theta(\mcA_0)) $
     whose kernel congruence is trivial, since $\theta([b_1]) = \theta([b_2])$ if and only if $\theta(b_1)=\theta(b_2),$  if and only if $b_1=b_2$; afterwards we embed $\theta(\mcA)$ into $\mcA',$ followed by embedding $\theta(\mcA_0)$ into $\mcA'_0$.
 \end{proof}

 \begin{example}
     Let us see how this corresponds to classical homomorphisms of algebras. The homomorphism $\varphi:\Z_{16}[\la] \to \Z_4[\la]$ given by  $f(\la) \mapsto [2f(\zero)]$ is the composite of the natural projection $\Z_{16}[\la] \to \Z_{4}$ given by $f\mapsto [2f(0)],$ followed by the natural injection $\Z_{4} \to \Z_{4}[\la]$ In the context of pairs, we could start with $(\Z [\la], 16\Z [\la]) \to (\Z  , 16\Z )$ given by $f\mapsto f(0),$ and then to the paired injection
     $(\Z  , 16\Z )\to (\Z [\la] , 32\Z [\la]),$ given by $m\mapsto 2m,$  followed by the extension of the null part $32\Z  \to 8\Z .$

 \end{example}\begin{MNote}\label{decomp1}
     Proposition~\ref{decomp} indicates how ideals fit into the the structure theory of pairs, via extensions of the null part; much remains to be done.
 \end{MNote}

\subsubsection{Property~N}\label{propN1} $ $

A key issue is how $\tT$ interfaces with $\mcA_0.$

\begin{definition}{\cite[\S3.1]{AGR2}}\label{propN0}$ $
A  pair
$(\mcA,\mcA_0)$ satisfies
 \textbf{Property~N} (for ``negation'')  if there is an invertible element $\one^\dag\in \tT$  such that,  defining $b^\dag = \one^\dag b$ and $b^\circ = b+ b^\dag$:
 \begin{enumerate}\eroman
     \item $e:=\one+\one^\dag\in \mcA_0$.
     \item (Uniqueness of $e$) If $a + a' \in \mcA_0$ for $a,a'\in \tT,$ then $a' =(-) a.$
   \item  $b^\circ \in \mcA_0$ for each $b\in \mcA$.
 \end{enumerate}
\end{definition}

 (iii)  is automatic for semiring pairs, and for $b\in \tT$.  When the  pair
$(\mcA,\mcA_0)$ is of the first kind, we may (and shall) take $\one^\dag =\one.$

 \begin{lem}[{\cite[Lemma~3.10]{AGR2}}]\label{est}  $\one^\dag e=  e.$ Thus, in a semiring pair, $e^2 = e+e.$
\end{lem}
\begin{MNote}\label{ex1}$ $ \begin{enumerate}\eroman

      \item
      The  pair $(\mcA,\{\zero\})$ of Note~\ref{UA}(iii) has  Property~N if and only if  $(\mcA,+)$ is a group. In particular  $(\mcA,+)$  lacks Property~N when $\mcA$ is idempotent (and thus ZSF),    and that   difficulty   provided the motivation for this study.

       \item {\it From now on, the pair  $(\mcA,\mcA_0)$    is assumed to
satisfy   Property~N.} We arbitrarily choose $\one^\dag$, which need not  be uniquely determined.  One   theme of this exposition is that the uniqueness of $e$ often is all that is needed.
      \item
In earlier incarnations, $(\one^\dag)$ was  designated $(-)\one$, with the extra requirement $((-)\one)^2 = \one,$ yielding a  ``negation map'' $b \mapsto ((-)\one)b .$  In this case we write $(-)b$ for~$b^\dag,$ and we write $b_1(-)b_2 $ for $b_1 + ((-)b_2).$  The pair
$(\mcA,\mcA_0)$ is \textbf{uniquely negated} if $(-)\one$ is unique.

   \item  Obviously the classical pair is uniquely negated.

\end{enumerate}
\end{MNote}

\medskip  \subsubsection{Surpassing relations and reversibility}$ $

We  need some relation to replace equality. Lorscheid~\cite{Lor1} used a symmetric relation; the one we adapt from \cite{Row22}, on the contrary,    is antisymmetric on $\tT.$

 \begin{definition}$ $
\begin{enumerate}\eroman
 \item   A \textbf{pre-order} on a $\tT$-module $\mathcal A$, denoted
  $\preceq$, is a set-theoretic pre-order that respects the $\tT$-module structure, i.e.,  for all   $b,b_i \in \mcA$:
  \begin{enumerate}
      \item  For all $a\in \tT$, $b_1\preceq b_2$ if and only if $a b_1\preceq a b_2 .$
         \item   $b_i \preceq b_i'$ implies $b_1 + b_2 \preceq b_1'+b_2'.$
         \end{enumerate}
    \item  A \textbf{pre-surpassing relation} on a pair $(\mathcal A,\mcA_0)$, denoted
  $\preceq$, is a pre-order on $\mcA$ satisfying the condition
\begin{itemize}
\item  If
$b\in \mcA_0$ then $ \zero \preceq b. $
\end{itemize}

\item We write $b_1 \succeq b_2$ if $b_2\preceq b_1;$ then
  $\succeq$ is also a pre-surpassing relation.
   \item  A  \textbf{surpassing relation} on a pair $(\mathcal A,\mcA_0)$ is a  pre-surpassing relation satisfying the condition
\begin{itemize}
\item
$a_1 \preceq a_2 $ for $a_1 ,a_2 \in   \tTz$ implies $a_1 =a_2.$

\end{itemize}

\item For any pair $(\mathcal A,\mcA_0)$,   a pre-{surpassing relation} $\preceq$   is  \begin{enumerate}
    \item  \textbf{  $\tT$-reversible} if $a_0^\dag \preceq \sum_{i=1}^n a_i$ implies $a_1^\dag \preceq  a_0 +\sum_{i=2}^n a_i$,

 \item \textbf{strongly $\tT$-reversible} if  $a+b\in \mcA_0$ implies $ a ^\dag \preceq b$,
\end{enumerate}
for $a, a_i\in\tT$, $b\in \mcA$.
\end{enumerate}
 \end{definition}

 \begin{example}\label{sur1}$ $
\begin{enumerate}\eroman
 \item  Every pair $(\mathcal A,\mcA_0)$ has the ``minimal'' pre-surpassing relation $\preceq_0$, given by defining $b_1 \preceq_0 b_2$ when $b_1 + b' = b_2$ for some $b' \in \mcA_0.$  In particular, for the classical pair, $\preceq_0$ is equality and is strongly $\tT$-reversible.

\item For a pair $( \mathcal A,\mcA_0)$, define the pre-surpassing relation  $b_1 \preceq  b_2 $ if $b_1^\circ +b_2^\circ = b_2^\circ.$ Then the classical pair satisfies $b_1^\circ = \zero= b_2^\circ$ so $b_1\preceq b_2$ for  {\it all} $b_i\in \mcA$.

 \item If $\tT =\{1\},$ then any pre-order on a semiring $\mcA$ (as in Note~\ref{UA}(iii)) is a surpassing relation on $(\mcA,\{0\})$. 

    \end{enumerate}
 \end{example}

 \begin{lem}
     If $\preceq$ is a surpassing relation then $\mcA_0\cap \tTz =\{ \zero\}$.
 \end{lem}
 \begin{proof}
     If $a\in \mcA_0\cap \tTz$ then $\zero \preceq a$, so $a= 0.$
 \end{proof}

\begin{lem}
    \label{rev1} If $b +b_1 \preceq b_2$ for $b,b_i\in \mcA$, then $b_1\preceq b^\dag + b_2.$
\end{lem}
\begin{proof}
    $b_1 \preceq b_1+ b^\circ = b^\dag+(b+b_1) \preceq b^\dag + b_2 $.
\end{proof}
\begin{lem}\label{rever1}
  If $(\mathcal A,\mcA_0)$ has   a strongly $\tT$-reversible pre-surpassing relation $\preceq$,\begin{enumerate}\eroman
       \item   $(\mathcal A,\mcA_0)$ is uniquely negated.

     \item   $\preceq$  is a $\tT$-reversible surpassing relation.
   \end{enumerate}
\end{lem}
\begin{proof} (i) Suppose $a' + a \in \mcA_0$. Then $\zero ^\dag \preceq a' +a, $ implying $a' \preceq \zero  +a^\dag ,$ so $a^\dag=a'$. By symmetry, $(a')^\dag =a,$ so $(a^\dag)^\dag =a,$ and $\dag$ is a negation map $(-)$, and $a' = (-)a$, so $(\mathcal A,\mcA_0)$ is uniquely negated.

(ii) If $a_1 \preceq a_2$ then $a_2 (-) a_1 \succeq a_1(-)a_1,$ implying $a_2=a_1 $ by (i). Hence $\preceq$  is a  surpassing relation.

Next, suppose $a_0^\dag \preceq \sum_{i=1}^n a_i\in \mcA_0$ for $a_i\in \tT$.
     Clearly $a_0 +  \sum_{i=1}^n a_i\in \mcA_0 $, implying $a_1^\dag \preceq a_0 +  \sum_{i=2}^n a_i\in \mcA_0,$ by hypothesis.
\end{proof}

Gunn~\cite{gu} recalls a property called  \textbf{fissure},  that if  $a_0 \preceq \sum_{i=1}^n a_i $ for $a_i\in \tTz$, then there is $a\in \tTz$ such that $a_0 \preceq a_1+a   $ and $ a \preceq
  \sum _{i=2}^n  a_i.$
\begin{lem}
    \label{fis}  Suppose $(\mcA,\mcA_0)$ is uniquely negated and satisfies  fissure. Then the surpassing relation $\preceq$ is strongly $\tT$-reversible.
\end{lem}
\begin{proof}  Suppose $\zero \preceq \sum_{i=1}^t a_i $ for $a_i\in \tT$. Fissure implies there exists $a$  such that $\zero \preceq a_1+ a $ and
    $a \preceq \sum_{i=2}^t a_i$. But then $a = (-)a_1$,   implying $(-)a_1 \preceq \sum_{i=2}^t a_i$.
\end{proof}

\begin{definition}
    Given a monoid pair $(\mcA,\mcA_0)$ (in the sense of Note~\ref{UA}(iv)) with a pre-surpassing relation~$\preceq$, the \textbf{$\preceq$-unit monoid} $\mcA^{\preceq \times}$ is $\{ b\in \mcA : \one \preceq b\}$, and
$\mcA^{\preceq \times}_0 = \mcA^{\preceq \times} \cap \mcA_0.$ 
\end{definition}
\begin{lem}
$\mcA^{\preceq \times}$ is a monoid, and  $(\mcA^{\preceq \times},\mcA_0^{\preceq \times})$ is a monoid pair.
\end{lem}
 \begin{proof}
     If $\one + c_i^\circ = b_i$ for $i=1,2$, then $\one + (c_1 b_2 + c_2b_1 + c_1c_2)^\circ =  b_1b_2.$
 \end{proof}

 \medskip

 \subsubsection{Metatangible and weakly metatangible pairs}$ $

The following kinds of pairs are especially malleable.

 \begin{definition}\label{metat}$ $
 \begin{enumerate}\eroman\item  $(\mcA,\mcA_0)$ is \textbf{weakly metatangible}
  if   $a_1+a_2 \in \tT \cup \mcA_0$ for any $a_1,a_2 $
in~$\tT.$

   \item ({\cite[Lemma~3.31]{AGR2} }) $(\mcA,\mcA_0)$ is  \textbf{metatangible} if it is   weakly metatangible and $\mcA = \tT +\mcA_0$.

   \item  As a special case, $(\mcA,\mcA_0)$ is \textbf{$\mcA_0$-bipotent}  if   $a_1+a_2 \in \{a_1,a_2\} \cup \mcA_0$ for any $a_1,a_2 $
in~$\tT.$
\end{enumerate}\end{definition}

\begin{theorem}[{\cite[Theorem~A]{AGR2}}]
\label{theoremA}For any metatangible pair $(\mathcal A, \mcA_0)$, every element $c$ of $ \mathcal A$
has a ``uniform presentation'', by which we mean there exists
$c_\tT \in \tT$ and $m_c\in \Net$, we have $ c_\tT \in \tT$,  such that
either: \begin{enumerate}\eroman \item  The pair $(\mathcal A, \mcA_0)$ is  of the first kind, with  $c = m_c c_\tT,$ or
  \item  The pair $(\mathcal A, \mcA_0)$ is  of the second kind, with one of the following:

\begin{enumerate} \item $c =\zero$, and $m_c=0.$  \item   $c = c_\tT \in \tT$ and $m_c=1$, or
  \item   $c = c_\tT^\circ$ and $m_c=2$.
\end{enumerate}
 \end{enumerate}
In (i),  for some $m_0 \ge 0,$  $\mcA_0 = \{ m_c c_\tT : m_c \ge m_0 \text{ or $m$ is even}\}.$

In (ii), $\mcA = \{\zero\}\cup \tT\cup \tT^\circ$, and $\mcA_0 = \tTz^\circ$.
  \end{theorem}

\begin{lem}\label{metrev} Suppose that a  uniquely negated  metatangible pair  $(\mathcal A, \mcA_0)$      has a surpassing relation $\preceq$. \begin{enumerate}\eroman
 \item If $\sum a_i \in \mcA_0$ for $a_i\in \tT$, then either $a_1=a_2=\dots =a_t$ or $(-)a_1 \preceq \sum_{i=2}^t a_i.$

    \item $\preceq$ is $\tT$-reversible.
\end{enumerate}
\end{lem}
\begin{proof}
(i) By induction on $t.$ If $t=2$ then $a_1+a_2\in \mcA_0,$ so $a_2 = (-) a_1.$ So may assume that $t\ge 3.$
If $a_{t-1}+a_t \in \tT,$ we can replace $a_{t-1}$ by $a_{t-1}+a_t$ and lower $t.$ Hence we may assume that $a_{t-1}+a_t \in \mcA_0,$ and thus $a_{t-1}= (-)a_t.$ Permuting the indices shows $a_{t-2}= (-)a_t$ and $a_{t-2}= (-)a_{t-1}.$ Hence $a_{t-2}= (-)a_{t-2}$.

(ii) Suppose $a_0'\preceq \sum_{i=1}^t a_i,$ with $t\ge 3.$ Then $\sum_{i=0}^t a_i \succeq a_0 +a_0'$ so is in $\mcA_0,$ and the assertion follow from (i).

\end{proof}

\begin{MNote}\label{hypst1}$ $
    \begin{enumerate}\eroman
        \item Lemma~\ref{metrev}  does not claim that   $\preceq$ is strongly $\tT$-reversible, since we could have $a+a+a,\ a+a \in \mcA_0.$ We do get strong reversibility   for weakly metatangible pairs in which $e+\one=e$.

         \item Metatangibility was assumed in much of \cite{Row24} as a key ingredient of the theory. But in this paper, motivated by Lemma~\ref{metrev}, we  use the weaker hypothesis of $\tT$-reversibility.
    \end{enumerate}
\end{MNote}

 \subsection{Semiring Pairs} $ $

\begin{definition}\label{srp}$ $
    \begin{enumerate}
    \item
A \textbf{semiring pair} $(\mcA,\mcA_0) $ is a pair for which $\mcA$ is a semiring.

\item
    A semiring pair  $(\mcA,\mcA_0)$  is a \textbf{paired domain} if  $(a_1+a_2)b\in \mcA_0$ for $a_i\in \tT$ implies $a_1+a_2 = a_1^\circ$   or $ b_2\in \mcA_0$.

  \item $(\mcA,\mcA_0)$  is a \textbf{strongly paired domain} if  $b_1b_2\in \mcA_0$ implies $b_1 \in \mcA_0$ or $ b_2\in \mcA_0$.

 \end{enumerate}
\end{definition}


\begin{MNote}\label{clas}
 The idea is that in semiring pairs we ``see'' the  elements $\tT,$ but $\mcA_0$ controls much of the structure.
\end{MNote}
\begin{lem}
Every shallow semiring pair    is a strongly paired domain.
\end{lem}
\begin{proof}
 If $b_1,b_2\in \mcA\setminus \mcA_0$ then $b_1,b_2\in \tT,$ so $b_1b_2\notin \mcA_0.$
\end{proof}


\medskip

\subsubsection{Examples of semiring pairs}\label{ex11} $ $

Let us see how to put all of our previous definitions and examples into the context of semiring pairs and pre-surpassing relations.
\begin{example}\label{ex11a}
    Notation is from \S\ref{na}.

\begin{enumerate}\eroman
 \item  Any semiring becomes a uniquely negated pair of the first kind, when we put $\tT = \{1\}$ and $\mcA_0 = \{ b+b:b\in \mcA\}$  As a special case we could take
$\mcA=\mathbb V_n$ (Example~\ref{gro}) and $(\mcA \cap 2\Net) \cup \{  n\}.$

       \item \textbf{Supertropical pairs}. Suppose $\tG$ is an ordered abelian  monoid with absorbing minimal element $\zero_{\tG}$, and $\tT$ is a monoid, together with a monoid homomorphism $\mu: \tT\to \tG$. The supertropical semiring $T(\tT;\tG)$  of \S\ref{st} is naturally a $\tT$-module under the given action.
   $(T(\tT;\tG),\tG)$ is the \textbf{supertropical pair}. It is shallow (as opposed to many other examples), and the identity map is a negation map of the first kind, providing  the ``flavor'' of characteristic~2.
  The supertropical pair is graded by the Boolean monoid $(\mathbb B,+) = \{0,1\}$; the $0$-component is $\tT$ and the  $1$-component is $\tG.$ Another way of describing $\preceq_0$ is $b_1\preceq_0 b_2$ for $b_2\notin \tT$ when $b_1+b_2=b_2.$

   There is a projection $\mu:T(\tT;\tG) \to \tG$, so the supertropical pair lies over the max-plus algebra of $\tG,$ and is a key non-classical example which has been studied in considerable detail.

   \item  Examples of pairs satisfying Property N which are not uniquely negated, taken from \cite{AGR2}:

    \begin{itemize}
   \item \cite[Example~2.21]{AGR2}  is the semiring pair $(\mcA ,\mcA_0 )$ with $\mcA = \tTz \cup \{e\}$, $\tT$~an arbitrary monoid, and $\mcA_0 = \{ \zero, e\},$
 satisfying $b_1+b_2= e$ for all $b_1,b_2 \in \tT \cup \{e\}.$ The pair $(\mcA ,\mcA_0 )$ is $\mcA_0$-bipotent and shallow and satisfies Property N, and is \textbf{terminal} in the sense that all pairs map to it by sending $b\mapsto e$ for all $b\ne \zero$.

     \item  A terminal idempotent $\mcA_0$-bipotent pair:
 Let $\mcA = \tTz \cup \{\infty\}$ where   $a+a = a$, but $a+a' =\infty$ for each $a\ne a'\in \tT$; then $\mcA_0 = \{\zero, \infty\}.$

   \item Suppose   $\mcA=\Net[\Lambda]$(of finite support) in commuting indeterminates $\{\la_{i }: 1\le i <\infty\}$  with underlying monoid  $\tT$ generated  by the $\la _i,$ we identify all elements  of $\mcA$ of support  of order at least 3 with $e$.  $(\mcA,\mcA_0)$ is not weakly metatangible since $\la_1+\la_2 \notin \tT + \mcA_0.$
    \end{itemize}

\item Any $\tTz$-semiring $\mcA$ gives rise to a  \textbf{doubled pair} with underlying monoid $(\tT \times \zero)\cup (\zero\times \tT),$ where $\hat A$ is as in Definition~\ref{doub}, and  the  null part $\hat \mcA_0$ is as in Remark~\ref{doubb}.  $(\hat \mcA,\hat \mcA_0)$ has the negation map $(-)(b_1,b_2)=(b_2,b_1)$ of the second kind, and the surpassing relation $\preceq_0$ described here by $(b_1,b_2)\preceq_0 (b_1',b_2')$ if there exists $c\in \mcA$ such that $b_i+c=b_i',$ $i=1,2$.

  \item \textbf{Paired tropical extensions}.   If $(\tG,+)$ is an ordered abelian semigroup and   $(\mathcal L,\mathcal L_0)$ is a semiring (resp.~nd-semiring) pair, then $(\mathcal L \rtimes \tG,  \mathcal L_\zero\rtimes \tG)$ is a  semiring (resp.~nd-semiring)  pair with underlying monoid $\tT\times \tG$,  under the action $(a,g)(\ell,g') = (a\ell,g+g')$ for $a\in \tT,$ $\ell \in \mathcal L,$ $g,g'\in \tG$.\footnote{We could perform the same construction for $(\mathcal L,\mathcal L_0)  = ( \{ \one, \infty\},\{\infty\})  $  as in Note~\ref{UA},   with $\one+\one = \infty$.    In this case   $(  \mathcal L \rtimes \tG,  (\mathcal L \rtimes \tG)_\zero)$ is isomorphic to the supertropical pair.}  $(  \mathcal L \rtimes \tG,  (\mathcal L \rtimes \tG)_\zero)$ is of the first kind, and is uniquely negated if $(\mathcal L,\mathcal L_0)$ is
uniquely negated.

   \item Suppose $(\mcA,\mcA_0)$ is a semiring pair. \begin{enumerate}
\item The \textbf{monoid semiring pair}. For a monoid   $ \mcM$, $(\mcA[\mcM],\mcA_0[\mcM])$ is a semiring pair  with underlying monoid $\tT\cdot\mcM = \{ay: a\in \tT,\ y\in \mcM\}$.    Any surpassing relation on $(\mcA,\mcA_0)$ passes to $(\mcA[\mcM],\mcA_0[\mcM])$, via $(f_u) \preceq (g_u)$ if $f_u\preceq g_u$ for all $u\in M.$ The special case $(\mcA[\Lambda],\mcA_0[\Lambda])$ is called the \textbf{polynomial pair}.

 \item The \textbf{function pair}.  $(\mcA^S,\mcA_0^S)$ is a semiring pair with underlying  monoid defined as the elements of $\tT^S$ having support of order 1.  Any surpassing relation on $(\mcA,\mcA_0)$ passes to $(\mcA^S,\mcA_0^S)$, via $f \preceq g$ if $f(s)\preceq g(s)$ for all $s\in S.$

   \end{enumerate}
   \end{enumerate}
\end{example}

   In (iv), (v), and (v), we have altered $\tT;$ this will be discussed more formally in \S\ref{morp1}.
 But then, even if  $(\mcA,\mcA_0)$ is metatangible, $(\mcA[\mcM],\mcA_0[\mcM])$ and $(\mcA^S,\mcA_0^S)$ and need not be weakly metatangible.

\subsection{Hyperpairs}$ $

Our other main example comes from hypergroup theory.
Any  hypersemigroup $\mcH$ is naturally a $\tT$-module, where $\tT =\mcH$, so $\mcA:=\nsets(\mcH)$ is a $\tT$-module, with the action viewed elementwise, and $\iota:
  \mcH\hookrightarrow \mcA$ is given by $\iota(a)=\{a\}$. We get a  pair $(\nsets(\mcH) ,\nsets(\mcH)_0)$ where $\nsets(\mcH)_0 = \{S \subseteq \mcH : \zero \in S\} .$
The pair  $(\nsets(\mcH) ,\nsets(\mcH)_0),$  called the \textbf{hyperpair} of $\mcH,$ is uniquely negated when $\mcH$ is a hypergroup.\footnote{In line with our generalization of negation map to Property~N,  one can weaken the condition of the hypernegative $-1$ so as not to be unique, but only requiring the set $e =1-1$ to be unique.}

The hyperpair $(\nsets(\mcH) ,\nsets(\mcH)_0)$ has the surpassing relation $\subseteq,$  which fits  very well  into this theory, as we see in the next two results.

\begin{lem} \label{metreva} The surpassing relation $\subseteq$ on a hyperpair satisfies fissure, and thus
     is strongly  $\tT$-reversible by  Lemma~\ref{fis}.
\end{lem}
\begin{proof}  Suppose $a_0 \preceq \boxplus _{i=1}^t a_i $ for $a_i\in \tT$. Then $a_0 \preceq a_1 \boxplus (a_2 \boxplus \dots \boxplus a_t),$ so there is some $a \in (a_2 \boxplus \dots \boxplus a_t)$ for which  $a_0 \in a_1\boxplus a, $ proving fissure.
\end{proof}

 ``Reversibility'' as defined in \cite{AGR2} is a special case of Lemma~\ref{metreva}.

 \begin{lem}\label{metrevc}
   Every  hyperpair   of a hyperfield  is a strongly paired domain.
 \end{lem}
  \begin{proof}
      If $\zero \notin S_1,S_2,$ then $\zero \notin S_1S_2 .$
  \end{proof}

\medskip
  $(\nsets(\mcH) ,\nsets(\mcH)_0)$
is weakly metatangible if and only $\mcH$ is ``stringent,'' cf.~\cite{BoS}.
 \subsection{Distributivity}\label{nddis}$ $

Since,   as in the   {phase hyperfield} of Example~\ref{hypers},  multiplication  in hyperpairs may fail to be distributive over addition, we  define
variants of distributivity.

\begin{definition}\label{dis1}
$ $
\begin{enumerate}
\item The analogous definitions from \ref{srp} hold for nd-semiring pairs.

    \item $(\mathcal A,\mcA_0)$ is $\preceq$-\textbf{distributive} with respect to a pre-surpassing relation $\preceq$, if $\sum_{i,j} a_i a_j' \preceq \sum_i a_i \sum_j a_j'$ for all $a_i,a_j'\in \tT.$
\end{enumerate}
\end{definition}
\begin{rem}\label{str}$ $
 \begin{enumerate}\eroman
     \item    Hyperpairs are $\subseteq$-{distributive}.
        \item  When the $\mcH$ is stringent, the hyperpair $(\nsets(\mcH) ,\nsets(\mcH)_0)$  is weakly metatangible and has   a shallow sub-hyperpair $(\nsets(\mcH)+\mcH, ,\nsets(\mcH)_0)$.
 \end{enumerate}
\end{rem}

\section{Roots of polynomials over pairs}

\begin{rem} $ $
\begin{enumerate} \eroman
\item Any paired  homomorphism $f: (\mcA,\mcA_0)\to (\mcA',\mcA'_0) $ extends naturally to a paired   homomorphism  $(\mcA^S ,\mcA  ^S_0) \to ({\mcA'}^S ,{\mcA'}  ^S_0)$, defined elementwise, as well as  to  a paired   homomorphism  $(\mcA[\mcM] ,\mcA  [\mcM]_0) \to ({\mcA'}[\mcM] ,{\mcA'}  [\mcM]_0)$.
    \item
    For any $\bfa = (a_i)\in \mcA^{(i)}$ there is a paired homomorphism  $$(\mcA[\Lambda] ,\mcA  [\Lambda]_0) \to (\mcA^{(I)} ,\mcA _0^{(I)})$$  where $\Lambda = \{\lambda_i: i \in I\},$  given by $\lambda_i \to a_i,$ $\forall i \in I.$
      \item (Special case for $|I|=1$)
    For any $a\in \mcA$ there is a paired homomorphism  $(\mcA[\lambda] ,\mcA _0 [\lambda]) \to (\mcA ,\mcA _0)$ given by $\lambda \to a,$ so $f\mapsto f(a)$ for all $f\in\mcA[\lambda] .$
\end{enumerate}
When $(\mcA,\mcA_0)$ is a semiring pair, these observations pertain to paired semiring homomorphisms.
\end{rem}

\begin{MNote}\label{roota}
    One basic tenet of classical algebra is that $a$ is a root of~$f$ if and only if $f \mapsto 0$ under the homomorphism $\mcA[\lambda] \to \mcA$
sending $\la \mapsto a.$
But for pairs, this must be modified to make sense for our main examples. To start with, $f(a) = \zero$ cannot hold for $f\ne \zero$ in a ZSF semiring\footnote{The way this is circumvented for semirings in general is to define a root of a pair $(f,g)$ of polynomials to be some $b\in \mcA$ such that $f(b)=g(b).$ This is precisely the definition of root in the symmetrized polynomial semiring.}.

Accordingly we have a different notion of root.
\end{MNote}

\begin{definition}\label{subs}  Suppose $\Lambda = \{ \la_i: i \in I\}.$ An element $\bfb\in \mcA^{(I)}$ is a \textbf{null root} of $f(\Lambda)$ if $f(\bfb)\in \mcA _0^{(I)},$ i.e., $\bfb \in f^{-1}(\mcA_0^{(I)})$.
\end{definition}

We call a polynomial \textbf{tangible} if its coefficients are in $\tT,$ despite the fact that the product of two polynomials need not be tangible.

\begin{rem}  The map $\mcA[\Lambda]\to \mcA$ given by $\Lambda \mapsto \bfa$ is a homomorphism if $\bfa\in \tT,$ which yields a
     paired homomorphism $(\mcA,\mcA_0)[\Lambda]\to (\mcA,\mcA_0) $, sending   $f$ to an element of $\mcA_0$ if and only if $\bfa$ is a null root of $f.$
\end{rem}
Our motivation for this definition is the Cayley-Hamilton Theorem (\ref{theoremG} below).

In classical algebra, an element $b$ is a root of $f$ if and only if $\lambda -b$ divides $f.$ Thus, we want an appropriate notion of  divisibility. The straightforward definition is not useful in general for polynomials and matrices, so we turn to alternatives.

\subsection{Factorization with respect to roots of polynomials}$ $

Throughout  this subsection, $(\mcA,\mcA_0)$ is a semiring pair\footnote{See \cite{Row26a} for the non-semiring generalization.}, and  $(\mcA[\lambda],\mcA[\lambda]_0)$ is the polynomial pair in the indeterminate $\lambda.$  We modify the treatment in \cite{Ho}.

 \begin{definition}  There are three kinds of factorization, for   $f_1,f_2$ in  $(\mcA^S,\mcA^S_0)$ or in $(\mcA[\Lambda],\mcA[\Lambda]_0)$.

  \begin{itemize}
   \item Usual division of elements in a monoid: $f_1 | f_2,$ if there is a tangible polynomial $ g$ such that $f_2 = g f_1 .$

      \item    $f_1$ $\preceq$-\textbf{divides} $f_2$, written $f_1 \, |_\preceq \, f_2,$ if  $f_2 \preceq g f_1 $  for some   tangible  $g $.

          \item    $f_1$ $\mcA_0$-\textbf{divides} $f_2,$   written $f_1 \ |_0 \ f_2,$ if  $f_2 + g f_1 \in \mcA_0$  for some   tangible~$g $.
  \end{itemize}
\end{definition}

(When $g$ is not required to be tangible, the definitions become vacuous.)

\begin{rem}\label{swd}
If $f_1 |_\preceq f_2,$ then $f_1 \ |_0 \ f_2.$

    Although division  is antisymmetric up to multiplication by a unit in a cancellative monoid,
    the situation    is trickier for polynomials over pairs, since the product of tangible polynomials need not be tangible. For example, $(\la +a)(\la(-)a) = \la^2 + (a(-)a)\la -a^2.$ So far we have not been able to circumvent this difficulty   in the arithmetic of polynomials. Nevertheless, there is a decent theory of their null roots.
\end{rem}

\begin{example}\label{exd} $ $
    $\lambda + b^\dag  \ |_\preceq \ \lambda^2 + (b ^2)^\dag$, because
    $$(\lambda + b^\dag )(\lambda +b) = \lambda^2 +(b ^2)^\dag +(b+ b^\dag )\lambda \succeq_0 \lambda^2  +b b ^\dag . $$
\end{example}

\begin{lem}
    Suppose $(\la -a) \ |_\preceq \ f,$  for a tangible polynomial $f$ i.e., $f \preceq (\la -a)g, $ for $g$ tangible. Then $\deg g = \deg f -1.$
\end{lem}
\begin{proof}
    The leading coefficient of $f$ is tangible, so must be matched by the  leading coefficient of $g$, with the powers of $\la$ matching.
\end{proof}

\begin{lem}\label{rt0} Given $b\in \mcA,$ let $g_{b,n} = \sum _{j=0}^{n-1}b^{n-j}\la^j. $
 \begin{enumerate}\eroman
 \item  $(\lambda + b^\dag)g_{b,n} = \la^n +(b^n)^\dag + \sum _{j=1}^{n-1}  (b^j+(b^j)^\dag )\la ^{n-j} $.

 \item $\la^n +(b^n)^\dag \preceq (\la+ b^\dag)g_{b,n}$.

  \item Suppose $f(\la) = \sum_{i=0}^n \alpha_i \la^i \in \mcA[\la].$ Let $g=  \sum _{j=1}^n\alpha_j g_{b,j} \in \mcA[\la].$ Then $$(\lambda + b^\dag)g = f(\la)  + f(b)^\dag+ \sum_{i=0}^n\sum_{j=1}^{i-1}   \alpha_i  (b^j+(b^j)^\dag )\la ^{n-j} .$$
  Hence $   f(\la)  + f(b)^\dag \preceq (\lambda + b^\dag)g $.
\end{enumerate}
\end{lem}
\begin{proof} (i) Direct computation.

(ii) By (i), $(\lambda + b^\dag)g_{b,n}= \la^n +(b^n)^\dag + \sum   (b^j+(b^j)^\dag )\la ^{n-j} \succeq \la^n +(b^n)^\dag $.

 (iii)
     $f(\la)  + f(b)^\dag = \sum \alpha_i \la^i + \sum \alpha_i (b^i)^\dag  = \sum \alpha_i (\la^i +(b^i)^\dag ), $ so $$(f(\la)  + f(b)^\dag)+ \sum_{i=0}^n\sum_{j=1}^{i-1}   \alpha_i  (b ^j+(b^j)^\dag )\la ^{n-j} =  \sum   \alpha_i(\lambda + b^\dag) g_i = (\la +b^\dag) g.$$

 (iv) By (iii) and Remark~\ref{swd}.

\end{proof}

  $f(\la)\preceq (\lambda + a^\dag)g $ since $f(a)\in \mcA_0.$ So to conclude that $(\lambda + a^\dag) \ |_\preceq \ f$ we   $g$ need to be tangible.

\begin{proposition}
    \label{rt} Suppose   $(\mcA,\mcA_0)$ is metatangible, and $f$ is a tangible polynomial for which $g$ of Lemma~\ref{rt0} is   tangible.
     Then $(\la+ a^\dag)\ |_\preceq \ f(\la)$.
\end{proposition}

\begin{proof}
 $(\lambda + a^\dag)g =  f(\la)  + f(a)^\dag+ \sum_{i=0}^n\sum_{j=1}^{i-1}   \alpha_i  (a^j+(a^j)^\dag )\la ^{n-j} ,$     by Lemma~\ref{rt0}(iii),  implying $f(\la)\preceq (\lambda + a^\dag)g $ since $f(a)\in \mcA_0.$
\end{proof}

\begin{rem} $(\la+ a^\dag)\ |_0 \ f(\la)$, under the hypothesis of Proposition~\ref{rt}.
The converse  is false in general, since  one can have $a+a(-)a \in \mcA_0$, for $a\in \tT,$ in which case $\la(-)a \, |_0 \, \la,$ and $a$ is not a null root of $\la.$

For a counterexample to the converse of Proposition~\ref{rt}, let $\mathcal{F}$ be the polynomial semiring $\Net[\lambda]$, and $\mcA$ be the doubled semiring of $\mcF [x_1,x_2,x_3]$, the polynomial semiring in three indeterminates, and take $$\mcA_0 =\diag + (x_1x_2 +x_1 x_3+x_2x_3)\mcA +(\sum x_i)\mcA.$$    $$f(\la) =  \la^3   ( -)x_1x_2x_3.$$
 $(\la (-) x_1)(\la (-) x_2)(\la (-) x_3) = \la^3 (-) \sum x_i \la^2 +(x_1x_3+x_2x_3+x_1x_2)\la (-)x_1x_2x_3 = f(-) \sum x_i \la^2+(x_1x_3+x_2x_3+x_1x_2)\la,$
implying $(\la -a_i)\ |_\preceq \ f(\la)$ for each $i.$ But $f(x_1) =    x_1(x_1^2(-)x_2x_3)\notin \mcA_0,$ since it has too few monomials. \end{rem}

\begin{definition}\label{prec2} Suppose $(\mcA,\mcA_0)$  is a pair with a negation map $(-)$ and a pre-surpassing relation $\preceq.$
An element    $a\in \tT$ is a  \textbf{factor-root} of $f$ if $f \preceq (\la(-)a)g$ for some tangible polynomial $g.$
\end{definition}

Gunn~\cite[Proposition~4.7]{gu} has an argument for null roots being factor-roots in a certain instance.
We modify his argument to provide a rather general fundamental result.
\begin{theorem}\label{root1}
 Suppose $(\mcA,\mcA_0)$   has a $\tT$-reversible surpassing relation $\preceq$ (which by Lemma~\ref{rever1} implies $(\mcA,\mcA_0)$ is uniquely negated), and $f\in \mcA[\la]$ with $a\in\tT$.
 \begin{enumerate}\eroman  \item Every factor-root $a$ of $f$   is a null root of $f$.
     \item If $\preceq$ satisfies fissure,    then every tangible null root of each tangible polynomial  is~a factor-root.
 \end{enumerate}
\end{theorem}
\begin{proof}
 (i) Suppose $f\, |_\preceq \,    (\la (-)a) g = \la  g (-)a g.$ Write $f = \sum \alpha_i \la^i$ and $g =\sum \beta_i \la^i.$ Matching coefficients we get $\alpha_i \preceq  (-) a\beta_{i} + \beta_{i-1} $, $i\ge 1,$
implying  $(-)\beta_{i-1}  \preceq (-)\alpha_i (-) a\beta_{i}$ since $\beta_{i-1}$ is presumed tangible, so
 $ \beta_{i-1}  \preceq  \alpha_i + a\beta_{i}$,  and $\alpha_0 \preceq (-)a \beta_0,$ implying
 $$\zero \preceq \alpha_0 +a \beta_0\preceq  \alpha_0 +a \alpha_1  +a ^2\beta_{1}\preceq  \sum _{i=0}^{n-1} \alpha_i a^i + a^n \beta_n = \sum _{i=0}^{n-1} \alpha_i a^i +\alpha_n a^n = f(a).$$

 (ii) Reversing the argument of (i), following the proof of \cite[Proposition~4.7]{gu}.
\end{proof}



 \subsection{Simultaneous roots}$ $

We  can treat several factor-roots simultaneously if $(\mcA,\mcA_0)$ is a  paired domain.

\begin{proposition}\label{sha} Suppose that a pair $(\mcA,\mcA_0)$ is an  paired domain with a $\tT$-reversible surpassing relation $\preceq$ satisfying fissure, such that either $\mcA$  is  a semiring  or $(\mcA,\mcA_0)$ is shallow. Then
    a monic polynomial $f$ cannot have more than $n$ distinct tangible factor-roots. If $a_1,\dots, a_n$ are  distinct tangible factor-roots then $f \preceq \prod_{i=1}^n (\la (-) a_i) .$
\end{proposition}
    \begin{proof}
        Writing $f \preceq(\lambda (-) a_1)g$ for $g$ monic, we have $\zero \preceq f(a_2) = a_2g(a_2) (-)a_1g(a_2). $

        If  $\mcA$  is  a semiring then $(a_2(-)a_1) g(a_2) \in \mcA_0,$ implying $g(a_2)\in \mcA_0.$

       If $(\mcA,\mcA_0)$ is  shallow and $g(a_2)\in \tT$, then $a_2g(a_2) (-)a_1g(a_2) = (a_2 (-)a_1)g(a_2),$ implying $a_2(-)a_1\in \mcA_0,$ so $a_1=a_2.$ Hence $g(a_2) \in \mcA_0,$ clearly of degree $n-1.$ Continuing inductively, and noting $\deg g = \deg f -1,$ we cannot have more than $n$ tangible factor-roots of $f$.
    \end{proof}

For application to hyperfields, see~Remark~\ref{str}.

\begin{rem}
This leaves us with multiple roots.
In classical algebra we say that $a$ is a double root of $f$ when $(\la (-) a)^2$ divides $f.$ Let us compare several options for generalization.
    \begin{enumerate}\eroman
    \item $a$ is a \textbf{double factor-root} of $f$ if $(\lambda (-) a)^2\, |_\preceq \, f,$ i.e., $(\la (-) a)^2 h \succeq f$ for some tangible polynomial $h$.

        \item (Like Gunn's version) $a$ is a \textbf{double factor-root} of $f$ if, for some tangible polynomial $g$, $f (-) (\lambda (-) a)g \in \mcA_0[\la]$ and $a$ is a factor-root of   $g$.

    \item $a$ is a \textbf{double factor-root} of $f$ if, for some tangible polynomial $g$, $f (-) (\lambda (-) a\preceq g$ and $a$ is a factor-root of   $g$.

         \item Define the formal derivative $f'$ of $f = \sum \alpha_i \la^i$ to be $\sum i \alpha_i \la^{i-1}.$
$a$ is a \textbf{double factor-root} of $f$ if $a$ is a factor-root of both $f$ and $f'$.
  \end{enumerate}
    These four versions are known to be equivalent in classical algebra, so which should we choose?
     Version (i) may seem the most natural at this stage, but we would like a double root also to be a single root. If  $(\la (-) a)^2 h \succeq f$ then clearly $(\lambda (-) a)g \succeq f $ for $g = (\la (-) a) h, $  but unfortunately $ (\la (-) a) h$ need not be tangible. This leads to (ii), in which Descartes' rule of signs has been proved for tropical extensions, in \cite{BaL,gu,AGT}, but there is ambiguity with respect to $g$, which leads to complications.  We could turn to   (iv), which is straightforward and unambiguous, but unfortunately $f'$ need not be tangible! But $f'$ indeed is tangible when no multiple of $a$ is in $\mcA_0$. We favor (iii), which is closest to a repetition of Definition~\ref{prec2}, and which we  use in \cite{Row26a}.

\end{rem}
\medskip

\section{Geometry}

\begin{definition} Suppose $\mcI$ is a $\tT$-submodule of   $\mcA.$ An
     $\mcI$-{\textbf{congruence pair}} $(\Cong,\Cong_\mcI)$ on $\mcA$ is a congruence $\Cong$ on $\mcA \times \mcA,$ together with its restriction $\Cong_\mcI$ to $\mcI\times \mcI.$

\end{definition}
\begin{rem}
    \label{conpr}$ $
    \begin{enumerate}\eroman
        \item

In particular, any congruence $\Cong$ can also be viewed as the $\mcA_0$-congruence pair $(\Cong,\Cong_0)$.

  \item Any paired map $\varphi:(\mcA,\mcA_0)\to (\mcA',\mcA'_0)$ yields, taking $\mcI = \varphi^{-1}(\mcA'_0)$, the $\mcI $-congruence pair $(\Cong_\varphi,{\Cong_\varphi}_\mcI)$, where $\Cong_\varphi := \{ (b_1,b_2): \varphi(b_1)= \varphi(b_2)\}.$
    \end{enumerate}
\end{rem}

\subsection{The
Zariski correspondence}$ $

Building on Definition~\ref{subs}(ii), given a set of polynomials $\mathcal I\subseteq \mcA[\Lambda],$ where  $\Lambda = \{ \la_i: i \in I\},$ define its \textbf{$\mcA_0$-locus}   $Z (\mathcal I)$ to be  $\{\bfb \in \mcA^{(I)}: f(\bfb)\in \mcA _0^{(I)}\}$.

There is a naive correspondence mimicking Zariski's classical correspondence, taking an ideal of $\mcA[\Lambda]$ to its $\mcA_0$-locus, with the reverse correspondence an ideal of $\mcA[\Lambda]$ to its null roots. But we want a correspondence to congruences, since they are more appropriate to semiring theory.

\begin{MNote}
    There is a $\tTz$-semiring paired homomorphism $$\varphi_\bfb  : (\mcA[\Lambda],\mcA[\Lambda]_0) \to  (\mcA^{(I)},\mcA^{(I)}_0) $$ defined by $\la_i \mapsto \bfb .$ Since, by Remark~\ref{conpr}(ii), the paired homomorphism  $(\varphi_\bfb,\varphi_{\bfb_0}):  (\mcA[\Lambda],\mcA[\Lambda]_0) \to  (\mcA^{(I)},\mcA^{(I))}_0) $ is defined in terms of an $\mcI$-congruence pair $\Cong_{\varphi_\bfb}$, where $\mcI = \varphi_\bfb^{-1}(\mcA'_0),$ we can define the congruence $\congg(Z) = \cap_{\bfb \in Z} \Cong_{\varphi_\bfb} .$

Conversely, any congruence $\Cong$ on $(\mcA[\Lambda],\mcA[\Lambda]_0)$ defines its $\mcA_0$-\textbf{set} $\{ b \in \mcA : f(b)=g(b) \in \mcA_0, \ \forall (f,g)\in \Cong\}.$
Clearly $Z(\congg(Z)) = Z.$
\end{MNote}

\begin{definition} $(Z,\congg)$ is the \textbf{Zariski correspondence} of pairs.
    A congruence $\Cong$ of $(\mcA[\Lambda],\mcA[\Lambda]_0)$ is \textbf{radical} if $\congg(Z(\Cong)) =\Cong.$
\end{definition}

\subsection{Spectrum of prime congruences}$ $

Having identified $A_0$-loci with congruences, one is led to study the spectrum of prime congruences. Joo and Mincheva \cite{JM} came up with the following  definition, which we state for  arbitrary  nd-semiring pairs.

\begin{definition}\label{twist2}  Suppose that $(\mcA,\mcA_0)$ is an nd-semiring pair, and $\Cong_i$ are congruences, viewed as subsets of $\widehat{\mathcal A}$, cf.~Definition~\ref{doub}.
\begin{enumerate}\eroman
\item
The \textbf{twist product} $\Cong_1 \ctw \Cong_2 := \{ \bfb_1 \ctw \bfb_2 : \bfb_i\in \Cong_i \}.$

\item
A congruence $\Cong$ of $\mcA$ is \textbf{semiprime} if it satisfies the following   condition: For a congruence $\Cong_1 \supseteq  \Cong$, if $\Cong_1 \ctw \Cong_1  \subseteq \Cong$  then $\Cong_1  = \Cong$.

\item
A congruence $\Cong$ of $\mcA$ is \textbf{prime} if it satisfies the following  condition: For  congruences  $\Cong_1$, $\Cong_2 \supseteq  \Cong$, if $\Cong_1 \ctw \Cong_2 \subseteq \Cong$, then $\Cong_1
= \Cong$ or $\Cong_2 = \Cong$.
\item
A congruence $\Cong$ of $(\mcA,\mcA_0)$ is \textbf{irreducible} if
whenever $\Cong_1 \cap \Cong_2 = \Cong$ for  congruences $\Cong_1$,
$\Cong_2 \supseteq  \Cong$, then $\Cong_1  = \Cong$ or $\Cong_2 =
\Cong$.

\end{enumerate}
\end{definition}

\begin{rem} $ $ \begin{itemize}
   \item The intersection of
semiprime congruences is  semiprime.

    \item A congruence $\Cong$ is prime iff $\Cong $ is semiprime and irreducible.
\end{itemize}
\end{rem}

 In classical algebra, Levitzki proved that any semiprime  ideal $\Cong$ is the intersection of
prime ideals,
which is a cornerstone in the theory of the prime spectrum.

\begin{lem}
    \label{commpr}
\protect {(As in \cite[Lemma~4.12]{Row25})} Suppose   $S\subseteq \widehat{A}$
is any subset disjoint from  $\Cong,$ satisfying  the property that for any congruences $\Phi_1,   \Phi_2 \supseteq \Phi,$ if $S\cap\Phi_1, S \cap \Phi_2$ are nonempty then $S \cap (\Phi_1\ctw \Phi_{2})\ne \emptyset.$  Then  there is
a congruence
$\Cong'\supseteq \Cong$ of $\mcA$ maximal with respect to being disjoint from~$
S$, and $\Cong'$ is prime.
\end{lem}

\begin{prop}
Any semiprime  congruence $\Cong$ of a  semiring is the intersection of
prime congruences.\end{prop}
\begin{proof}
  Following Levitski, define a set $S\subset\widehat{A},$   by the  following inductive method:

  For any
   $s_1 \in \widehat{A} \setminus \Cong,$ and given $s_i\in S$ take $\bfb \in \mcA$ and $s_{i+1}\in S$ such that $s_{i+1} \ctw \bfb \ctw s_{i+1} \notin \Cong.$ By Lemma~\ref{commpr}, we have some prime congruence disjoint $\Cong'$ from $S$. Hence the intersection of all such prime congruences misses all elements of~$\widehat{A} \setminus \Cong, $ so is $\Cong.$
\end{proof}

However, even in the idempotent commutative case, it is tricky to describe the  intersection of all
prime congruences, cf.~\cite[Example~3.3]{JM}.
Joo and Mincheva  prove that the prime congruence spectrum of $\Rmax[\la_1,\dots,\la_n]$ is catenary of length~$n.$
This result is extended  in \cite{Row25} to semiring pairs satisfying the property  $\one + ke = ke$ for some $k > 0$. Although this property fails for   classical pairs, it holds in Example~\ref{ex11a} (i), (ii), and (iii), all finite hyperpairs, and of course is preserved under polynomials and functions.
There is much room for further investigation.
\begin{ques}
    How far can one describe the prime spectrum? Is there an intrinsic description of the intersection of all prime congruences? Is there a ``weak Nullstellensatz'' describing semiprime congruences?
\end{ques}

 \section{Matrix pairs}

\begin{definition}\label{mapr}
    The \textbf{matrix pair} $M_n(\mcA,\mcA_0)$ of a semiring pair $(\mcA,\mcA_0)$ is $(M_n(\mcA),M_n(\mcA_0))$, with underlying partial monoid $\cup _{1\le i,j\le n}\tT e_{i,j}$.
\end{definition}

   \subsection{Singularity of matrices}\label{det1a}
$ $

    We call a matrix \textbf{tangible} if all of its entries are in $\tT.$
A \textbf{track} of an $n \times n$ matrix
$ A = (a_{i,j})\in M_n(\mcA)  $ is a product $a  _\pi := a_{ \pi(1),1 } \cdots a_{ \pi(n),n }$ for $\pi\in S_n,$ where either $A$ is tangible so that the tracks are in $\tTz$, or $\mcA$ is a semiring;
in either situation, the  following formulas 
make sense.

\begin{equation}\label{eq:tropicalDetsignsyma}
 {\Det A} _+ = \sum _  {\pi  \in S_n \text{ even }} a_{\pi}, \qquad {\Det A} _- =
 \sum _ {\pi  \in S_n \text{ odd }}   a_{\pi}.
\end{equation}

Also \eqref{eq:tropicalDetsignsyma} makes sense over semiring pairs.
This allows us to define singularity of matrices  as follows:

  \begin{definition}     The matrix $A$ is  \textbf{singular}\footnote{In \cite{AGR2} one defines a {\it balance relation} $\nabla$ and singularity with respect to $\nabla$, but this definition is more encompassing.}
  if ${\Det A} _+ = {\Det A} _-$. 
  \end{definition}



\subsection{$\dag$-determinants}$ $

Reutenauer and Straubing \cite{ReS},  and
Zeilberger \cite{zeilberger85} showed that a number of determinantal
identities of matrices admit bijective proofs, in the sense that the excess tracks pair off. This was the key idea in \cite[\S 6]{AGR2}, and is used again in the results given in this section and the next. The issue here is to avoid $(\one^\dag)^2,$ since it may not be $\one.$

Suppose  $\one + \one^\dag = e$ in the pair $(\mcA,\mcA_0)$. We define the \textbf{$\dag$-determinants} $|A|_\dag = {\Det A} _+ + {\Det A} _-^\dag$ and $_\dag |A|= {\Det A} _+^\dag + {\Det A} _-$.  Familiar results can be reformulated for pairs:

\begin{lem}\label{det1}
  For $A,B\in M_n(\mcA)$,\begin{enumerate}\eroman
    \item  $|AB|_\dag \succeq_0 |A|_\dag |B|_\dag $.    \item  $_\dag|AB| \succeq_0 {_\dag}|A| {_\dag}|B| $.
\end{enumerate}
\end{lem}
\begin{proof}
    The excess tracks on the left side pair off, contributing elements of $\mcA_0.$
\end{proof}
  \cite[Theorem~E]{AGR2} is generalized here to:

\begin{theorem}[Cayley-Hamilton theorem]\label{theoremG}
 For $A \in M_n(\mcA),$ viewed in $\widehat{\mcA}^+$ as in Remark~\ref{doubb} define the \textbf{characteristic polynomial} $f_A(\lambda) =|(\zero,\lambda I) + (A,\zero)|_\dag.$ Then $f(A)\in \widehat{\mcA}_0.$
\end{theorem}
\begin{proof}
    By \cite[Theorem~E]{AGR2}, applied to the doubled pair
$(\hat \mcA,\hat \mcA_0).$
\end{proof}

Also applying  Proposition~\ref{sha} to  \cite[Theorem~E]{AGR2} yields

    \begin{cor}\label{HCa}
        The characteristic polynomial $|\la I (-) A|$ of a tangible matrix~$A$ over a shallow paired domain with unique negation cannot have more than $n$ factor-roots.
    \end{cor}

\subsection{$\dag$-adjoints}

 \begin{definition}
    \label{adjo}   Write $A_{i,j}$ for
  the $(j,i)$ minor of a matrix $A$. Define $a_{i,j}'$ as $|A_{i,j}|_\dag$ for $i+j$ even, and $_\dag|A_{i,j}|$ for $i+j$ odd. The
\textbf{$\dag$-adjoint} matrix $\adj _\dag A$ is the matrix $(a_{i,j}')$.
 \end{definition}

 We reformulate \cite[Theorem~F]{AGR2}.

\begin{theorem} [Generalized Laplace identity and Cauchy-Binet formula]
\label{theoremH} $A = (a_{ij})$ satisfies
Laplace's well-known identity $$|A| = \sum _{j=1}^n
a'_{i,j}a_{i,j},$$ for any  $i$.

For the more general  Cauchy-Binet formula, fix $I = \{ i_1, \dots, i_m \}\subset \{1, 2, \dots ,
n\}.$  For any  $J = \{ j_1, \dots, j_m\}\subset \{1, 2, \dots , n\}$,    write $A_{I,J}$ for the  $(n-m)\times (n-m)$
minor obtained by deleting all rows from $I$ and all columns from~$J$. Take $a_{I,J}'$ for $|A_{I,J}|_\dag $     when $(i_1 +\dots + i_m)(j_1 +\dots + j_m)$ is even, and for $_\dag|A_{I,J}| $     when $(i_1 +\dots + i_m)(j_1 +\dots + j_m)$ odd. Then \begin{align}\label{e-Lap1}|A| = \sum _{J : |J|= m}  a'_{I,J}|a_{I,
J}|.\end{align}
\end{theorem}
\begin{proof}
    The proof goes over mutatis  mutandis, since the formula was arranged so that $\dag$ does not repeat.
\end{proof}

\subsection{Eigenvalues of matrices} $ $

We define eigenvalues by means of the surpassing relation $\preceq.$
\begin{definition}
    An \textbf{eigenvector} of a  matrix $A$ is  a vector $\bfv$ such that $A\bfv \succeq \alpha \bfb$; and $\a$ is called the corresponding  \textbf{eigenvalue}.
\end{definition}

\begin{lem}
    If $\bfv_j$ are eigenvectors of $A$ with eigenvalues $\a_j,$ and $B$ is the matrix whose $j$ column is $\bfv_j,$ then $AB \succeq BD,$ where $D = \diag \{ \a_1, \dots, \a_n\}.$
\end{lem}
\begin{proof}
    By definition of eigenvector.
\end{proof}

Let us bring in
Theorem~\ref{theoremG}.

\begin{theorem}
    The eigenvalues of $A$ are the null roots of the characteristic polynomial $f_A(\lambda).$ If  these null roots are distinct, then $AB \succeq BD,$ where the columns of $B$ are $\mcA_0$-independent.
\end{theorem}

This coordinates with Corollary~\ref{HCa}, when factor-roots are null roots, and gives us a technique for diagonalization.


\medskip  \subsection{Matrix semigroups }$ $

 The next natural step is to define the analogs of the classical linear algebraic groups. This is largely open, but some work   done on $\SL_n$ \cite{INR} can be done in this context.

 \begin{definition}
 $\SL(n,(\mcA,\mcA_0)) =   \{A\in M_n(\mcA): {\Det A} _\dag \succeq \one\} $.
 \end{definition}

 \begin{lem}
      $\SL(n,(\mcA,\mcA_0)) $ is a monoid.
 \end{lem}
\begin{proof}
    If ${\Det A_i} _\dag \succeq 1$ for $i=1,2,$ then ${\Det A_1A_2} _\dag \succeq {\Det A_1} _\dag {\Det A_2} _\dag  \succeq 1.$
\end{proof}
We get monoid pairs, letting  $\SL(n,(\mcA,\mcA_0))_0 $ to be the singular matrices in~$\SL(n,(\mcA,\mcA_0))$.
\medskip

\subsubsection{Involutions}\label{inv}$ $

\begin{definition}
Suppose $\mcA$ is an nd-semiring.

\begin{enumerate}
    \item The \textbf{opposite nd-semiring}  $\mcA^{\op}$ has the same addition as $\mcA,$ but opposite multiplication $(b_1\cdot b_2) = b_2b_1$.

    \item An \textbf{anti-homomorphism} on $\mcA$ is a homomorphism from $\mcA$ to $\mcA^{\op}$, i.e., satisfying $f(b_1b_2) =  f(b_2)f(b_1),$ $\forall b_i\in \mcA.$
    \item An \textbf{involution}  on an nd-semiring pair $(\mathcal
A, \mcA_0)$  with   surpassing relation~$\preceq$, is a ``paired'' anti-homomorphism of degree 2, i.e.,  satisfying ($\forall  \ b,b_i \in  \mathcal
A$):

 \begin{enumerate}\eroman

\item $(b_1b_2)^* = b_2^*b_1^*,$
 \item $(b^*)^* = b$,
 \item $\mcA_0^*\subseteq \mcA_0,$
 \item
 $(b^\dag)^*= ( b^*)^\dag,$ implying $(b^\circ)^*= ( b^*)^\circ,$
 \item $\tT^* = \tT$,
 \item If $b_1 \preceq b_2$
in $\mathcal A,$ then $b_1^* \preceq b_2^*.$

\end{enumerate}
\end{enumerate}
\end{definition}

\begin{lem}  $(c^\circ)^* =
(c^*)^\circ.$
\end{lem}
\begin{proof} $(c^\circ)^* = (c+c^\dag)^*  = c^* + (c^\dag)^* =
(c^*)^\circ.$
\end{proof}

\begin{example}\label{matrixinv}  $ $ Examples of involutions on the matrix semialgebra $M_n({\mathcal
A})$:
\begin{enumerate}\eroman
   \item The transpose map on $M_n({\mathcal
A})$ is an involution denoted by $A \mapsto A^t.$

\item  When $n = 2m$, there is another involution, called the
\textbf{symplectic} involution $(s)$, given by $\left(\begin{matrix}
                A_{11} & A_{12}^\dag\\
                A_{21} &  A_{22}
               \end{matrix}\right)^s = \left(\begin{matrix}
               A_{22}^t &  A_{12}^t\\
                {A_{21}^\dag}^t &  A_{11}^t
               \end{matrix}\right)$, for $A_{ij}\in M_m( \mathcal
A).$
               \end{enumerate}
\end{example}
%

 \begin{MNote}
      We can define $\preceq$-orthogonal
matrices via the condition $(I,(\zero))\preceq A A^t, A^tA , $ and
thereby define the $\preceq$-orthogonal monoids, special
$\preceq$-orthogonal monoids, and $\preceq$-symplectic
monoids. These concepts would be special cases of $\preceq$-algebraic groups, defined in the natural way as varieties of pairs endowed with a quasi-inverse map $(g g^{-1} \succeq I)$, but have not yet been investigated.
 \end{MNote}

    Right multiplication by an element of a finite monoid can be viewed as a permutation matrix in the usual way. Since a permutation matrix has a single track, its $\dag$-determinant is tangible, in fact $\one$ or $\one^\dag.$ However, to have a meaningful character theory, one needs to take traces.

When $A = (a_{i,j})$, write $\tr(A) $ for $\sum a_{i,i}.$

 \begin{rem}
   $  \tr( AB) = \tr( BA).$ Hence $\{A \in M_n(\mcA,\mcA_0): \tr (A)\succeq \zero\}$ produces a Lie pair \cite{GaR}.
 \end{rem}

\begin{ques}
   How does one work out a character theory of monoid pairs?
\end{ques}

\section{Vector space pairs and linear algebra}

\begin{definition}\label{vs}
   A \textbf{vector space pair} over a semiring pair $(\mcA, \mathcal A _0)$  is a pair $  (\mcV,\mcV_0),$ where $\mathcal V :=\mathcal A ^{(J)}$  and $\mcV_0 =\mcA_0^{(J)}$.    (Usually $J = \{1,\dots, n\}.$) The $\tT$-bimodule operations are defined componentwise, as is the surpassing relation.

\end{definition}
Different notions of dependence are given in
\cite[Chapter ~19]{golan92} and \cite{AGG1}, but as one might expect, we focus on the ones arising from pairs and surpassing relations.

  \begin{lem}\label{dsum}
    If $\mathbf{v_n} = \sum_{i=1}^{n-1}\mathbf{v_i},$ then the matrix $A$ whose rows are $\mathbf{v_1},\dots, \mathbf{v_n}$ is  singular.
\end{lem}
\begin{proof}
  The $+$ and
$-$ parts in \eqref{eq:tropicalDetsignsyma} match.
\end{proof}

\begin{definition}\label{dep1}
 A set of
  vectors $\{\mathbf v_i \in
\mcV : i \in I\}$ is   $\mcV_0$-\textbf{dependent}, if $\sum _{i\in I'} a_i
\mathbf v_i \in \mathcal V_0$ for some nonempty finite subset $I' \subseteq I$ and $a_i \in \tT$.

\end{definition}

 $\mcV_0$-dependence is studied in terms of matrices, as in classical linear algebra, via the following conditions.

\begin{itemize}\label{Cond}
\item \textbf{Condition \A{1}}:
The submatrix rank (i.e., the maximal size of a nonsingular square submatrix) is less than or equal to the row rank and the
column rank.

  \item \textbf{Condition \A{2}}:  The submatrix rank is greater than or equal to the row rank and the
column rank.

\item   \textbf{Condition \A{3}}: If the row rank  of an $n\times  n$ matrix~$A$ is~$n$, then $A$ is nonsingular.

\item   \textbf{Condition \A{4}}, cf.~\cite[Question~5.13]{BaZ}: The rows of an $m\times n$ matrix are $\mcV_0$-dependent if $m>n.$

\item   \textbf{Condition \A{5}}: For $m>n,$ either the rows of an $m\times n$ matrix $A$ are $\mcV_0$-dependent or    $A$   has a singular $n\times n$ submatrix.

 \end{itemize}

 \begin{rem}\label{imps}$ $
  \begin{enumerate}\eroman
      \item   Each of these conditions  hold for supertropical algebras \cite{IR}.

    \item   Each of these conditions    has a counterexample in \cite{AGR2}.
There are straightforward
counterexamples to Conditions \A{2} and \A{3}.

    \item  Condition~\A{1}  holds in varied situations, including
    \begin{itemize}
        \item matrices   over a metatangible, uniquely negated  pair \cite[\S 7.3.1]{AGR2}.

          \item quotient hyperfields, by the argument in \cite{BaZ} of passing down from the field.

            \item  Various other hypotheses   yielding Condition \A{1} are given in \cite[\S 7.3 and \S 7.5]{AGR2}.
    \end{itemize}

    \item  Condition \A{2} implies \A{3}.

    \item  Condition \A{3} is seen to hold in certain cases, in \cite[Theorems~O,P,Q]{AGR2}.

       \item  Condition~\A{3} implies \A{4}, since adding on   $m-n$ columns of zeroes gives a singular matrix, whose rows must thus be $\mcV_0$-dependent.

          \item     Condition \A{4} holds in many situations, cf.~\cite[\S 7.7]{AGR2},  including all hyperfield examples given in \cite[\S 7.7.1]{AGR2}.
  \end{enumerate}
 \end{rem}

When $\mcA$ has a pre-surpassing  relation, the pre-surpassing relation
 $\preceq$  on vectors in $V$ is defined componentwise, i.e., $(v_j) \preceq (v_j')$ if $v_j \preceq v_j'$ for each $i$.
Although $\mcV_0$-dependence is our predominant tool, the following notion also is useful.
 \begin{definition}
     A vector $\mathbf v$ is $\preceq$-\textbf{spanned} by vectors $\mathbf v_1,
\dots, \mathbf v_n$ if  there exist  $a_i \in \tTz$ for which $\mathbf v  \preceq  \sum_{i=1}^n a_i \mathbf v_i .$
 \end{definition}

 \begin{rem}
     For a $\tT$-reversible pre-surpassing relation  $\preceq$, if $\mathbf v_1,\dots, \mathbf v_n$ are tangible and $\mathbf v_1 $ is $\preceq$-spanned by  $\mathbf v_2,\dots, \mathbf v_n$, then $\mathbf v_2 $ is $\preceq$-spanned by  $\mathbf v_1,\mathbf v_3,\dots, \mathbf v_n$.
 \end{rem}

When $(\mcA, \mathcal A _0)$ is shallow, Condition \A{5} holds by {\cite[Proposition 7.4]{AGR2}}. In fact, by \cite[Theorem~U]{AGR2}, either $A$ has a singular $n\times n$ submatrix or each row is $\preceq_0$-spanned by the others.

One of the main tools for solving linear equations is Cramer's rule, which is related to Condition 1 and studied in depth in \cite{AGR2}.

\section{Varieties and polynomial identities}

It is natural to study varieties of pairs, either from the point of view of congruences or noncommutative $\mcA_0$-loci, leading to the notion of polynomial identity:

\begin{definition} Suppose $f(X_1,\dots,X_n)$, $g(X_1,\dots,X_n)$ are tangible polynomials in noncommuting indeterminates over a semiring pair $(\mcA,\mcA_0)$.
\begin{itemize}
    \item $f(X_1,\dots,X_n)= g(X_1,\dots,X_n)$ is a \textbf{polynomial identity} (PI) of a semiring $\mcA$ if $f(b_1,\dots,b_n)= g(b_1,\dots,b_n)$ for all $b_i\in \mcA;$

    \item $f(X_1,\dots,X_n) $ is an $\mcA_0$-\textbf{PI} of  $(\mcA,\mcA_0)$ if $f(b_1,\dots,b_n)\in \mcA_0$ for all $b_i\in \mcA.$
\end{itemize}
\end{definition}

\begin{rem} For any semiring pair $(\mcA,\mcA_0)$,
    if $f=g$ is a PI of  $\mcA$, then $f+g^\dag$ is an $\mcA_0$-{PI} of $(\mcA,\mcA_0)$.
\end{rem}

By viewing any commutative semiring as a  homomorphic image of a free commutative semiring $\Net[\Lambda]$, which in term is contained in the free commutative ring~$\Z[\Lambda]$,    Akian, Gaubert and Guterman \cite{AGG1}~observed that any polynomial expression holding in  all commutative rings also holds in all commutative semirings. By passing to matrix semirings, and matching components, sorting positive and negative parts, it follows that all PI's of $M_n(\Z)$ yield PI's of $M_n(\mcA)$ for any commutative semiring $\mcA,$ and thus $\mcA_0$-PIs of  $M_n(\mcA,\mcA_0).$  \cite{AGG1} also contains a  ``strong transfer principle,'' building on Reutenauer and  Straubing~\cite{ReS}.   This raises the ``Specht-type'' question,
\begin{ques}
Do all $\mcA_0$-PIs of $M_n(\mcA,\mcA_0)$ follow formally from finitely many $\mcA_0$-PIs?
\end{ques}

  \section{Categories: Morphisms of pairs}\label{morp1}

So far the only morphisms we have discussed are homomorphisms in the sense of universal algebra.
A   subtle  categorical issue   is to determine the most useful  definition of morphism     of pairs for applications.   The theory flows most easily when we use the universal algebra approach with paired homomorphisms, as seen in \cite{Row24a}.
Pairs   are closed under products and sums in universal algebra.

However, there are two  other definitions of morphisms which cast better light on    many non-classical pairs.

\begin{definition}\label{wm0} All maps $f:(\mcA,\mcA_0)\to (\mcA',\mcA'_0)$ are presumed to satisfy $f(\mcA_0)\subseteq \mcA_0'$ and $f(\zero) = \zero'.$ We distinguish according to how $f$ acts  additively.
    \begin{enumerate}
     \item\label{weak-morphism} $f$ is a \textbf{weak morphism of pairs}  if, for all $n\in \Net,$  $\sum_{i=1}^n b_i \in \mcA_0$ implies $\sum _{i=1}^n f(b_i)\in \mcA'_0$, for $b_i\in \mcA$.
     \item
$f$ is a $\preceq$-\textbf{morphism}  with respect to   pre-surpassing relations $\preceq$ of $(\mcA,\mcA_0)$ and $ (\mcA',\mcA'_0)$,  if $f( b ) \preceq   f(b')  $
for  $b \preceq b '$   and
$f(\sum_{i=1}^n b_i) \preceq \sum_{i=1}^n f(b_i)  $  for all $n\in \Net,$ $b_i\in \mcA$.



   \item     A \textbf{weak morphism} (resp.~ $\preceq$-\textbf{morphism}) of  semiring pairs (or, more generally,  nd-semiring pairs) is a weak morphism (resp.~$\preceq$-{morphism}) which also is a multiplicative map.

     \item  $\succeq$-\textbf{morphisms} are defined analogously.
    \end{enumerate}
\end{definition}

\begin{lem}\label{wkm}(As in {\cite[Lemma~2.10]{AGR1}})
    Every  $\preceq$-{morphism} or $\succeq$-{morphism} of a pair with respect to a pre-surpassing relation is a weak morphism.
\end{lem}

\begin{rem}
  The $\dag$-determinant on matrices is a $\succeq_0$-morphism, by Lemma~\ref{det1}.

\end{rem}

\begin{example}
    For any pair $(\mcA,\mcA_0)$ and any $a\in \tT,$ the left multiplication map $y\mapsto ay$ is a paired homomorphism. When $(\mcA,\mcA_0)$ is a $\preceq$-distributive nd-semiring pair, $L_b$ is a $\preceq$-homomorphism.
    When $(\mcA,\mcA_0)$ is a semiring pair, the left multiplication map $L_b : y\mapsto by$ is a paired homomorphism for all $b\in \mcA$.
\end{example}

\begin{lem} Take $\tT = \tG$ in \S\ref{st}, with   $\val: \mathbb{K} \to  \tT$  the Puiseux valuation.  Then the composition $-\val: \mathbb{K} \to  \tT $ yields a $\preceq$-morphism $(\mathbb{K},\{0\}) \to  \mathbb (T(\tT;\tG),\tG)$.
\end{lem}
\begin{proof}
    $-\val (f+g) = \max \{-\val f ,-\val g\} $, which is the sum in $\mathbb T(\tT;\tG),$ unless $\val f = \val g$,  in which case $-\val (f+g) \le  -\val f \preceq_0 \mu( -\val f),$ which is  the sum of $  -\val f$ and~$ -\val f $ in $\mathbb T(\tT;\tG)$.
\end{proof}

\begin{rem}
    \label{depb} Let $(\mcV,\mcV_0)$ be a vector space pair.
    \begin{enumerate}\eroman
        \item The left multiplication map by     $a\in \tT$ preserves dependence of vectors if and only if it is a weak morphism of $(\mcV,\mcV_0)$.
        \item The left multiplication map by      $a\in \tT$  preserves $\preceq$-dependence of vectors if and only if it is a $\preceq$-morphism of $(\mcV,\mcV_0)$.
    \end{enumerate}
\end{rem}

The terminology for   morphisms in hyperfields vary (and find their origin in~\cite{Dr}). Here we say a map $f: \mcH \to \mcH'$ of hypergroups is a \textbf{weak morphism} if $0 \in \boxplus a_i$ implies $0 \in \boxplus f(a_i)$; $f$ is a \textbf{strong morphism} if $f(\boxplus a_i) \subseteq \boxplus f(a_i).$

 \begin{lem}\label{ext1} There is a functor from  hypergroups to  hyperpairs, sending a map $f: \mcH \to \mcH'$ of hypergroups   to the map   $\hat f: (\nsets(\mcH),\nsets(\mcH_0)) \to  (\nsets(\mcH'),\nsets(\mcH'_0))$ given by $\hat f (S) = \{f(a):  a\in S\}.$

  \begin{enumerate}\eroman
     \item    If $f$ is a weak morphism of hypergroups then  $\hat f $  is a weak morphism of hyperpairs.
      \item    If $f$ is a strong morphism of hypergroups then  $\hat f $  is a $\subseteq$-morphism .
  \end{enumerate}
 \end{lem}
 \begin{proof}
     Clearly if $S_1 \subseteq S_2$ then $f(S_1)\subseteq f(S_2),$ and \begin{equation}
         \begin{aligned}
             \hat f(S_1\boxplus S_2) & = \{ f(a): a \in S_1 \boxplus S_2\} = \{f (a_1\boxplus a_2): a_i \in S_i\}\\& \subseteq  \{f (a_1)\boxplus f(a_2): a_i \in S_i\} =  \hat f(S_1)\boxplus \hat f(S_2).
         \end{aligned}
     \end{equation}
     If $\zero\in S$ then $\zero = f(\zero) \in f(S).$
 \end{proof}
\begin{rem}
    The \textbf{direct product} $\prod_{i\in I}(\mcA_i,{\mcA_i}_0)$ of pairs $(\mcA_i,{\mcA_i}_0)$ having underlying monoid $\tT_i$, is defined as the pair $(\prod \mcA_i,\prod {\mcA_i}_0)$ having underlying monoid $\prod\tT_i$, under the product action $(a_i)(b_i)= (a_ib_i).$ This construction also is functorial.
\end{rem}

\subsection{Extensions of pairs: tensor extensions}\label{ext}$ $

So far, the underlying monoid $\tT$ has been fixed in our discussion of pairs. But the polynomial and matrix constructions motivate changing the underlying monoid.

\begin{definition}
    An \textbf{extension} of a pair
    $(\mcA,\mcA_0) $ is a $\tT'$-pair   $(\mcA',\mcA_0') $ with $\tT'\supseteq \tT$, such that     there is a paired injection  $(\mcA,\mcA_0)\to  (\mcA',\mcA_0')$.
\end{definition}

The ``paired tropical extension,'' polynomial pair, and function pair  of Example~\ref{ex11a}(v),(vi) are examples.
Here is a general class of extensions.

\begin{definition}

When $\mcA$ is a $\tT$-module, and  $\tT'$ is an abelian monoid containing~$\tT,$ we can define the \textbf{tensor extension} $\tT'$-module $\mcA\otimes _\tT \tT'$ to be   the $\tT'$-module $\mcS(\tT' \times \mcA)/\Cong,$
where $\mcS(\tT' \times \mcA)$ is the free semigroup having base $\tT' \times \mcA$, made into a  $\tT'$-module by the action on the first component, and $\Cong$ is the
congruence
 generated by all
 \begin{equation}
    \label{defcong1a} \bigg((  a x ,  v ), (x ,av
)\bigg),
\end{equation}\begin{equation}\label{defconga}\bigg(\big(x, v+w\big),  \big( x,v  )+(x,w \big)\bigg), \qquad   \forall x\in \tT',\, v ,w \in \mathcal \mcA_i, \,  a \in
\tT, \end{equation}
 cf.~\cite[\S4.3]{Row24a}. When $(\mcA , \mcA _0)$ is a pair, $(\mcA\otimes _\tT \tT', \mcA_0\otimes _\tT \tT' )$ is a pair with underlying monoid $\tT'$, called the $\tT'$-\textbf{tensor pair extension}.
 Homomorphisms, weak morphisms and $\preceq$-morphisms all lift  up tensor pair extensions, by \cite[Theorem~5.15]{Row24a}.

\end{definition}
\begin{example} Tensor pair extensions relate to our previous constructions of~extensions.

\begin{enumerate}\eroman
    \item $(\mcA[\lambda],\mcA_0[\lambda]) \cong (\mcA ,\mcA_0)\otimes\tT'$ where $\tT' = \{a_i\la^i: i \in \Net\}.$

     \item $(\mcA^S,\mcA_0^S) \cong (\mcA ,\mcA_0)\otimes\tT'$ where $\tT'  $ is  the submonoid of   elements of $\tT^S$ having support of order 1.

       \item If $\tT$ is contained in a group $\tT'$ then $ (\mcA ,\mcA_0)\otimes\tT'$ is a pair over the group~$\tT'$. This is an easy instance of localization.
\end{enumerate}

\end{example}

 \begin{ques} What is a good general theory of integral extensions?
 \end{ques}

 \section{Module  pairs}

Motivated by the study of rings,  we would like to use modules (often called ``semimodules'' in the literature)  analogously in the study of semiring pairs. The theory of projective \cite{Ka} and injective modules over semirings is carried out in the usual categorical setting using homomorphisms, also cf.~\cite[Chapter ~17]{golan92}. This has led  to the search for a homology theory of semirings \cite{Tak,KN}.

 \begin{definition} Assume that $\mcA$ is a $\tTz$-semiring, and  that $\mcM$ is a $\tT$-module satisfying $(ay)b = a(yb) = y(ab)$ for all $a\in \tT,$ $b\in \mcA,$ $y\in \mcM.$ Also assume that there is an action $\mcA \times \mcM$, which restricts to the given $\tT$-action.
    \begin{enumerate}\eroman
        \item $\mcM$ is a \textbf{module} over $\mcA$ if  $\sum_{i,j} b_i y_j = (\sum_i b_i)(\sum_j y_j)$,  for all  $b_i\in \mcA,$ $y_j\in \mcM.$

    \item $\mcM$ is a   $\preceq$-\textbf{module}  over $\mcA$ with respect to a surpassing relation $\preceq$ if   $\sum_{i,j} b_i y_j \preceq (\sum_i b_i)(\sum_j y_j)$,   for all  $b_i\in \mcA,$ $y_j\in \mcM.$


    \item  A \textbf{module  pair} over a semiring pair $(\mcA,\mcA_0)$ is an abstract pair $(\mcM,\mcM_0)$ such that   $\mcM \supseteq \mcM_0$ are modules over $\mcA$ and $\mcA_0 \mcM \subseteq \mcM_0.$ $\preceq$-\textbf{module pairs} are defined analogously.
    \end{enumerate}
\end{definition}

\begin{rem}
  Any hyperpair is a $\subseteq$-module pair over itself.
\end{rem}

\medskip

\begin{MNote}
    Representation theory for modules depends on  the choice of definition of morphism in Definition~\ref{wm0}, which   depends in turn on the context.
\end{MNote}

\subsubsection{Tensor products of module pairs}$ $

Next, we take a special case of \cite{Row24a}, which considers left and right $\tT$-modules.

 \begin{definition} \label{tp1}$ $
\begin{enumerate}\eroman
  \item Suppose $\mcA_1,\mcA_2$ are $\tT$-modules. Define the \textbf{tensor product $\tT$-module} $\mathcal A_1  \otimes _{\tT} \mathcal A_2$ to be the semigroup $\mcS(\mcA_1 \times \mcA_2)/\Cong,$
where $\Cong$ is the
congruence
 generated by all
 \begin{equation}
    \label{defcong1a1} \bigg(( a x_1 , x_2 ), (x_1,a
 x_2
)\bigg),
\end{equation}\begin{equation}\label{defconga1}\bigg(\big( v_1+w_1,x_2\big),  \big( v_1, x_{2})+(w_1,  x_{2}\big)\bigg), \end{equation}\begin{equation}\label{defcongb}\bigg(\big( x_{1},   v_2+w_2\big),   (x_{1},v_2 ) +(x_{1},w_2 )  \bigg) ,\end{equation}

  $ \forall x_{i}, v_i,w_i\in \mathcal \mcA_i, $ $   a \in
\tT$.


\item When $(\mcA_i,{\mcA_i}_0)$  are pairs, define
$$(\mcA_1\otimes_{\tT} \mcA_2)_0 := ({\mcA_1}_0\otimes_{\tT} \mcA_2)+ (\mcA_1\otimes_{\tT} {\mcA_2}_0).$$

\end{enumerate}

$((\mcA_1\otimes_{\tT} \mcA_2),(\mcA_1\otimes _{\tT}\mcA_2)_0)$ is a pair,  called the \textbf{tensor pair}.
 \end{definition}


By \cite[Corollaries~[4.13,4.15]{Row24a}, the tensor product satisfies the usual associativity and distributivity over direct sums. 

 Tensor products of module pairs are seen in \cite[Corollaries~4.15, 4.16]{Row24a} to provide a concrete tensor category with respect to homomorphisms, in the sense of \cite{Jag,Po}, and satisfy the adjoint correspondence \cite[Remark~5.17]{Row24a}.    Tensor products of weak morphisms and $\preceq$-morphisms are considerably subtler, treated in \cite[\S 5]{Row24a}.

\medskip  \subsection{ $\preceq$-Projective  and $\preceq$-injective  modules}

\begin{definition}[\cite{JMR1}]\label{pproj}$ $
\begin{enumerate}\eroman
    \item A $\preceq$-morphism $f : (\mcM,\mcM_0) \to  (\mcN,\mcN_0)$  is $\preceq$-\textbf{cofinal} if  for every $y \in N$, there exists $b \in M$ for which
 $y \preceq f (b)$.

       \item  A  $\preceq$-module pair $(\mcP,\mcP_0)$    is $\preceq$-\textbf{projective} if for any $\preceq$-cofinal $\preceq$-morphism $h : (\mcM,\mcM_0) \to  (\mcN,\mcN_0)
,$ every $\preceq$-morphism $f :(\mcP,\mcP_0)\to (\mcM,\mcM_0)$ lifts
to a $\preceq$-morphism $\tilde f :(\mcP,\mcP_0)\to (\mcM,\mcM_0)$, in the sense that $f=h \tilde f.$
\end{enumerate}
\end{definition}

  The $\succeq$-version is analogous.
The $\preceq$-projective theory is developed in \cite{JMR1}, to include a ``dual basis lemma'' and a ``Shnauel's lemma.'' In \cite{JMR2} this is carried forward towards homological theory with mixed success: A partial ``connecting lemma'' is obtained in \cite[Theorem~5.15]{JMR2}. Much work remains to obtain a useable $\preceq$-homology theory.

{Localization} via injective modules is treated in \cite[Chapter~18]{golan92}. One can dualize Definition~\ref{pproj}, to obtain $\preceq$-injective modules and   the $\preceq$-injective hull.

\begin{ques}
   How does one develop  a $\preceq$-quotient theory?
\end{ques}

\end{document}